\documentclass[reqno,11pt]{article}
\usepackage{a4wide}
\usepackage{color,eucal,enumerate,mathrsfs,amsthm}
\usepackage[normalem]{ulem}
\usepackage{amsmath,amssymb,amsthm,epsfig,bbm,makecell}
\numberwithin{equation}{section}

\usepackage{amsfonts}
\usepackage{dsfont}
\usepackage{mathtools}
\usepackage{leftidx}
\usepackage{graphicx}
\usepackage{esint}
\usepackage{authblk}
\usepackage{multirow}

\usepackage[pdfborder={0 0 0}]{hyperref}  

\usepackage[latin1]{inputenc}
\usepackage[english]{babel}

\newtheorem{Definition}{Definition}[section]
\newtheorem{Proposition}[Definition]{Proposition}
\newtheorem{Lemma}[Definition]{Lemma}
\newtheorem{Theorem}[Definition]{Theorem}

\theoremstyle{definition}

\newtheorem{Remark}[Definition]{Remark}

\allowdisplaybreaks

\newcommand{\R}{\mathbb{R}}

\newcommand{\mm}{{\mbox{\boldmath$m$}}}

\newcommand{\ggamma}{{\mbox{\boldmath$\gamma$}}}

\newcommand{\sggamma}{{\mbox{\scriptsize\boldmath$\gamma$}}}

\newcommand{\sfd}{{\sf d}}

\newcommand{\sfh}{{\sf h}}

\newcommand{\Kliminf}{K\kern-3pt-\kern-2pt\mathop{\rm lim\,inf}\limits}  
\newcommand{\supp}{\mathop{\rm supp}\nolimits}   
\newcommand{\Lip}{\mathop{\rm Lip}\nolimits}          
\renewcommand{\d}{{\mathrm d}}
\newcommand{\dt}{{\d t}}
\newcommand{\ddt}{{\frac \d\dt}}

\newcommand{\restr}[1]{\lower3pt\hbox{$|_{#1}$}} 

\newcommand{\la}{\left<}                  
\newcommand{\ra}{\right>}
\newcommand{\eps}{\varepsilon}  
\newcommand{\nchi}{{\raise.3ex\hbox{$\chi$}}}

\newcommand{\fr}{\penalty-20\null\hfill$\blacksquare$}                      

\newcommand{\adm}{{\rm{Adm}}}                   
\newcommand{\prob}[1]{\mathscr P(#1)}                   
\newcommand{\probt}[1]{\mathscr P_2(#1)}                   

\renewcommand{\mm}{\mathfrak m}                                

\renewenvironment{proof}{\removelastskip\par\medskip   
\noindent{\em proof} \rm}{\penalty-20\null\hfill$\square$\par\medbreak}

\newcommand{\testipp}[1]{{\rm Test}^{\infty}_{>0}(#1)}

\newcommand{\X}{{\rm X}}

\newcommand{\h}{{\sfh}}

\renewcommand{\ae}{{\textrm{\rm{-a.e.}}}}

\newcommand{\RCD}{{\sf RCD}}

\newcommand{\lip}{{\rm lip}}

\newcommand{\HS}{{\lower.3ex\hbox{\scriptsize{\sf HS}}}}

\newcommand{\Y}{{\rm Y}}
\newcommand{\hr}{{\sf r}}
\newcommand{\hR}{{\sf R}}

\title{Benamou-Brenier and duality formulas for the entropic cost on $\RCD^*(K,N)$ spaces}
\author{Nicola Gigli \thanks{SISSA, Trieste. email: ngigli@sissa.it} \quad Luca Tamanini \thanks{Institut f\"ur Angewandte Mathematik, Universit\"at Bonn. email: tamanini@iam.uni-bonn.de}}

\begin{document}

\maketitle

\begin{abstract}
In this paper we prove that, within the framework of $\RCD^*(K,N)$ spaces with $N<\infty$, the entropic cost (i.e.\ the minimal value of the Schr\"odinger problem) admits:
\begin{itemize}
\item[-] a threefold dynamical variational representation, in the spirit of the Benamou-Brenier formula for the Wasserstein distance;
\item[-] a Hamilton-Jacobi-Bellman dual representation, in line with Bobkov-Gentil-Ledoux and Otto-Villani results on the duality between Hamilton-Jacobi and continuity equation for optimal transport;
\item[-] a Kantorovich-type duality formula, where the Hopf-Lax semigroup is replaced by a suitable `entropic' counterpart.
\end{itemize}
We thus provide a complete and unifying picture of the equivalent variational representations of the Schr\"odinger problem (still missing even in the Riemannian setting) as well as a perfect parallelism with the analogous formulas for the Wasserstein distance.
\end{abstract}

\tableofcontents

\section{Introduction}

It is well-known that the optimal transport problem with quadratic cost admits two equivalent formulations: the Kantorovich dual one \cite{Kantorovich42} and the Benamou-Brenier dynamical one \cite{BenamouBrenier00}. Given two compactly supported probability measures $\mu_0 = \rho_0\mathcal{L}^d,\mu_1 = \rho_1\mathcal{L}^d$ in $\R^d$, the former tells us that the squared Wasserstein distance between $\mu_0$ and $\mu_1$ can be represented as
\begin{equation}\label{eq:kantdual}
\frac{1}{2} W_2^2(\mu_0,\mu_1) = \sup_{\phi \in C_b(\R^d)}\int \phi\,\d\mu_0 + \int\phi^c\,\d\mu_1
\end{equation}
with $\phi^c$ the $c$-conjugate of $\phi$, namely
\[
\phi^c(x) := \inf_{y \in \R^d}\frac{|x-y|^2}{2} - \phi(y).
\]
On the other hand, J.-D.\ Benamou and Y.\ Brenier observed that the optimal transport problem admits a fluid-dynamics interpretation in the following sense:
\begin{equation}\label{eq:smoothbb}
W_2^2(\mu_0,\mu_1) = \inf_{(\rho,v)}\iint_0^1 |v(x,t)|^2\rho(x,t)\,\dt\d x
\end{equation}
where the infimum runs over all couples $(\rho,v)$ solving the continuity equation
\[
\partial_t\rho + {\rm div}(\rho v) = 0, \qquad \textrm{in } \R^d \times (0,1)
\]
with marginal constraints $\rho(\cdot,0) = \rho_0$ and $\rho(\cdot,1) = \rho_1$ in $\R^d$.

Although the two results may appear completely different in spirit at a first glance, as a matter of fact they are deeply linked, since it is possible to pass directly from the one to the other, gaining a further variational representation of the Wasserstein distance. This relies on the duality between Hamilton-Jacobi and continuity equation, hidden in \eqref{eq:smoothbb} and already noticed by Benamou and Brenier. Indeed, the Hamilton-Jacobi equation 
\[
\partial_t\phi + \frac{1}{2}|\nabla\phi|^2 = 0, \qquad \textrm{in } \R^d \times (0,1)
\]
already arises in the optimality conditions for $v$ in \eqref{eq:smoothbb}, which read as $v = \nabla\phi$ with $\phi$ solving the PDE above. More generally, if $(\phi_t)$ is a subsolution to the Hamilton-Jacobi equation and $(\rho,v)$ is a solution to the continuity equation, then
\[
\int\phi(\cdot,1)\,\d\mu_1 - \int\phi(\cdot,0)\,\d\mu_0 \leq \frac{1}{2}\iint_0^1 |v(x,t)|^2 \rho(x,t)\dt\d x
\]
and, as first noticed by F.\ Otto and C.\ Villani in \cite{OttoVillani00} and by S.\ Bobkov, I.\ Gentil and M.\ Ledoux in \cite{BoGeLe01}, if we saturate the left-hand side with the supremum and the right-hand one with the infimum, then the inequality turns out to be an equality, which in particular yields
\begin{equation}\label{eq:hjdual}
\frac{1}{2} W_2^2(\mu_0,\mu_1) = \sup\int\phi(\cdot,1)\,\d\mu_1 - \int\phi(\cdot,0)\,\d\mu_0
\end{equation}
where the supremum runs over all subsolutions to the Hamilton-Jacobi equation with initial condition $\phi(\cdot,0) = \phi_0$ for all possible $\phi_0 \in C_b(\R^d)$. Now it is sufficient to recall that maximal subsolutions to the Hamilton-Jacobi equation are obtained via the Hopf-Lax semigroup $Q_t$, defined for any $f : \R^d \to \R \cup \{+\infty\}$ and $t>0$ as
\[
Q_t f(x) := \inf_{y \in \R^d}\frac{|x-y|^2}{2t} + f(y).
\]
This allows us to restate \eqref{eq:hjdual} as
\begin{equation}\label{eq:kantdual2}
\frac{1}{2} W_2^2(\mu_0,\mu_1) = \sup_{\phi \in C_b(\R^d)}\int Q_1\phi\,\d\mu_1 - \int\phi\,\d\mu_0,
\end{equation}
which is equivalent to \eqref{eq:kantdual} by the very definition of the Hopf-Lax semigroup, thus closing the loop.

The great interest in the study of metric (measure) spaces (we refer to \cite{Heinonen07} for an overview on the topic and detailed bibliography) led subsequently the focus on the possible generalization of these results. The essentially metric nature of optimal transport with quadratic cost, of the Hopf-Lax formula and of the Hamilton-Jacobi equation (up to replace $|\nabla\phi|$ with $\lip\phi$) suggests that \eqref{eq:kantdual2} should hold on rather general metric spaces and this is actually the case, as shown in \cite{AmbrosioGigliSavare11} and \cite{GoRoSa14}. For the Benamou-Brenier formula, the Hamilton-Jacobi duality \eqref{eq:hjdual} and the continuity equation, on the contrary, the natural framework is the one of metric measure spaces and in \cite{Kuwada10}, \cite{Gigli-Kuwada-Ohta10} the duality between Hamilton-Jacobi and continuity equation is pointed out, whereas in \cite{GigliHan13} the non-smooth analogue of \eqref{eq:smoothbb} is established. Adopting the language developed by the first author in \cite{Gigli14}, we first say that a curve $(\mu_t)\subset \probt\X$ is a solution of the continuity equation
\[
\ddt\mu_t+{\rm div}(X_t\mu_t)=0,
\]
where $t \mapsto X_t \in L^0(T\X)$ is a family of vector fields, possibly defined only for a.e.\ $t \in [0,1]$, provided it is weakly continuous, the map $t \mapsto \int |X_t|^2\d\mu_t$ belongs to $L^1(0,1)$, $\mu_t \leq C\mm$ for all $t \in [0,1]$ for some $C>0$ and for any $f\in W^{1,2}(\X)$ the map $[0,1]\ni t\mapsto\int f\,\d\mu_t$ is absolutely continuous with
\[
\ddt\int f\,\d\mu_t=\int \d f(X_t)\,\d\mu_t\qquad {\rm a.e.}\ t.
\]
With this premise, if $(\X,\sfd,\mm)$ is infinitesimally Hilbertian and $\mu,\nu \in \probt\X$ are such that there exists a $W_2$-geodesic $(\mu_t)$ connecting them such that $\mu_t \leq C\mm$ for all $t \in [0,1]$ for some $C>0$ (the statement in \cite{GigliHan13} is slightly more general), then
\[
W_2^2(\mu,\nu) = \min\int_0^1\int |X_t|^2\d\mu_t\dt,
\]
where the minimum is taken among all solutions $(\mu_t,X_t)$ of the continuity equation such that $\mu_0 = \mu$ and $\mu_1 = \nu$.

\bigskip

Formally similar to optimal transport but with completely different motivation and interpretation, the Schr\"odinger problem is an optimization and interpolation problem too. While optimal transport was originally formulated by G.\ Monge for engeneering purposes, such as resource allocation, on the contrary Schr\"odinger problem is physical in nature, as in trying to explain the wave-particle duality via a classical mechanics example, E.\ Schr\"odinger landed in a maximal likelihood problem. In both cases two probability measures $\mu_0,\mu_1$ are assigned as data, but while in optimal transport they are seen as initial and final configurations of resources whose transportation cost has to be minimized among all possible couplings $\ggamma$ between $\mu_0$ and $\mu_1$, in Schr\"odinger problem $\mu_0$ and $\mu_1$ represent initial and final probability distributions of diffusive particles and one looks for the most likely evolution from $\mu_0$ to $\mu_1$. Let us briefly describe what this means in the Euclidean setting.

Given two probability measures $\mu_0=\rho_0\mathcal L^d$, $\mu_1=\rho_1\mathcal L^d$ on $\R^d$, one looks for a coupling between them that takes into account the fact that the particles are driven by a diffusion process. As shown in \cite{Follmer88} (see also \cite{Leonard14} for a detailed explanation), this amounts to solve the following minimization problem
\begin{equation}\label{eq:schrodinger}
\inf_{\sggamma \in \adm(\mu_0,\mu_1)} H(\ggamma\,|\,\hR)
\end{equation}
where $H(\cdot\,|\,\cdot)$ denotes the Boltzmann-Shannon entropy and $\d\hR(x,y) = \hr_{1/2}(x,y)\mathcal{L}^{2d}$, $\hr_{1/2}$ being the heat kernel at time $t=1/2$ (the time choice plays no special role, but is convenient in computations). It turns out that in great generality this problem admits a unique solution $\ggamma$ and the structure of the minimizer is very rigid: indeed $\ggamma = f \otimes g\,\hR$ for some Borel functions $f,g : \R^d \to [0,\infty)$, where $f\otimes g(x,y) := f(x)g(y)$. As a consequence
\[
\rho_0=f\,\h_{1/2}g\qquad\qquad\qquad\qquad\rho_1=g\,\h_{1/2}f,
\]
where $\h_tf$ is the heat flow starting at $f$ evaluated at time $t$. This suggests us to interpolate from $\rho_0$ to $\rho_1$ by defining
\[
\rho_t:=\h_{t/2}f\,\h_{(1-t)/2}g.
\]
This is called entropic interpolation, in analogy with displacement one. Introducing the Schr\"odinger potentials $\varphi_t,\psi_t$ (in connection with Kantorovich ones) as
\[
\varphi_t := \log \h_{t/2}f \qquad\qquad\qquad\qquad \psi_t := \log \h_{(1-t)/2}g,
\]
the parallelism between optimal transport and Schr\"odinger problem can be fully appreciated. Indeed, by direct computation it is not difficult to see that $(\varphi_t),(\psi_t)$ solve the Hamilton-Jacobi-Bellman equations
\[
\partial_t\varphi_t = \frac{1}{2}|\nabla\varphi_t|^2 + \frac{1}{2}\Delta\varphi_t \qquad\qquad\qquad\qquad -\partial_t\psi_t = \frac{1}{2}|\nabla\psi_t|^2 + \frac{1}{2}\Delta\psi_t,
\]
and they are linked to $(\rho_t)$ via the Fokker-Planck equations
\[
-\partial_t\rho_t + \rm{div}(\nabla\varphi_t\,\rho_t) = \frac{1}{2}\Delta\rho_t \qquad\qquad\qquad \partial_t\rho_t + \rm{div}(\nabla\psi_t\,\rho_t) = \frac{1}{2}\Delta\rho_t.
\]
Thus, denoting by $\mathscr{I}(\mu_0,\mu_1)$ the minimal value of \eqref{eq:schrodinger}, the entropic analogue of the Benamou-Brenier formula is
\begin{equation}\label{eq:intbbs}
\begin{split}
\mathscr{I}(\mu_0,\mu_1) & = H(\mu_0\,|\,\mathcal{L}^d) + \inf_{(\nu^+,v^+)}\int_0^1\int\frac{|v^+_t|^2}{2}\d\nu^+_t\dt \\ & = H(\mu_1\,|\,\mathcal{L}^d) + \inf_{(\nu^-,v^-)}\int_0^1\int\frac{|v^-_t|^2}{2}\d\nu^-_t\dt
\end{split}
\end{equation}
where the infimum is taken among all suitable weak solutions of the forward Fokker-Planck equation in the first case and of the backward one in the second case, with marginal constraints $\nu^\pm_0 = \mu_0$ and $\nu^\pm_1 = \mu_1$. If we also introduce the functions $\vartheta_t := \frac{\psi_t-\varphi_t}{2}$ it is not hard to check that it holds
\[
\partial_t\rho_t + \rm{div}(\nabla\vartheta_t\,\rho_t) = 0,
\]
and a third Benamou-Brenier formula for $\mathscr{I}(\mu_0,\mu_1)$ is available, namely
\begin{equation}\label{eq:intbbs2}
\mathscr{I}(\mu_0,\mu_1) = \frac{1}{2}\Big(H(\mu_0\,|\,\mathcal{L}^d) + H(\mu_1\,|\,\mathcal{L}^d)\Big) + \inf_{(\eta,v)}\bigg\{\iint_0^1\Big( \frac{|v_t|^2}{2} + \frac{1}{8}|\nabla\log\eta_t|^2 \Big)\eta_t\dt\d\mm\bigg\}
\end{equation}
where the infimum now runs over all suitable weak solutions of the continuity equation with marginal constraints $\eta_0\mathcal{L}^d = \mu_0$ and $\eta_1\mathcal{L}^d = \mu_1$. This has been first realised in \cite{Leonard14}, \cite{ChGePa16} and then extended to a slightly more general setting in \cite{GLR15}, moving from closely related results contained in \cite{Nelson67}, \cite{Follmer88}, \cite{Nagasawa93} and the subsequent literature. A heuristic discussion can be found also in \cite{Leger17}.

As concerns the entropic analogue of Kantorovich duality, the natural guess is then to replace solutions of the Hamilton-Jacobi equation with those of the Hamilton-Jacobi-Bellman equation in \eqref{eq:hjdual} and thus to substitute the Hopf-Lax formula with a suitable semigroup providing us with solutions of the latter PDE. This is given by
\begin{equation}\label{eq:semigroup}
\tilde{Q}_t \phi(x) := \log\big(\h_t e^\phi\big), \qquad \forall \phi \in C_b(\R^d)
\end{equation}
and thus \eqref{eq:hjdual} becomes
\begin{equation}\label{eq:inthj}
\mathscr{I}_\eps(\mu_0,\mu_1) = H(\mu_0\,|\,\mathcal{L}^d) + \sup\bigg\{\int \phi_1\,\d\mu_1 - \int \phi_0\,\d\mu_0 \bigg\},
\end{equation}
as shown in \cite{MikamiThieullen06}, where the supremum is taken among all supersolutions to the backward Hamilton-Jacobi-Bellman equation with final condition $\phi(\cdot,1) = \phi_1$ for all possible $\phi_1 \in C_b^\infty(\R^d)$, while \eqref{eq:kantdual2} turns into
\begin{equation}\label{eq:intks}
\mathscr{I}(\mu_0,\mu_1) = H(\mu_0\,|\,\mathcal{L}^d) + \sup_{\phi \in C_b(\R^d)}\bigg\{\int\phi\,\d\mu_1 - \int \tilde{Q}_1\phi\,\d\mu_0\bigg\},
\end{equation}
as proved in \cite{GLR15}. Both \eqref{eq:inthj} and \eqref{eq:intks} are forward representations and thus admit backward counterparts. Truth to be told, in \cite{MikamiThieullen06} and in the subsequent work \cite{MikamiThieullen08} the Schr\"odinger problem is not explicitly mentioned; nonetheless, a direct link between \eqref{eq:intbbs} and \eqref{eq:intks} is established.

As for optimal transport, also for the Schr\"odinger problem it is reasonable to investigate what can be said in the non-smooth setting. In fact, the construction of entropic interpolation and Schr\"odinger potentials can be done in great generality, as only a heat kernel is needed. In this sense, in the recent works \cite{GigTam17} and \cite{GigTam18} the authors brought the Schr\"odinger problem to finite-dimensional $\RCD^*(K,N)$ spaces, obtaining new (even in the Euclidean setting) uniform bounds for the densities of the entropic interpolations and the local Lipschitz constants of the Schr\"odinger potentials that will be recalled in Section \ref{sec:2.2}. 

Still, the generalization of \eqref{eq:intbbs}, \eqref{eq:intbbs2}, \eqref{eq:inthj} and \eqref{eq:intks} to the non-smooth setting has not been achieved yet, not even on smooth Riemannian manifolds (except for a partial result obtained for \eqref{eq:intks} in \cite{GRST17}, where Kantorovich duality for general transport costs is established in the metric setting and applies to the Schr\"odinger problem in the case the space is assumed to be compact). As concerns the Benamou-Brenier formulas for the entropic cost, this is essentially due to the fact that in \cite{MikamiThieullen06}, \cite{GLR15} and \cite{ChGePa16} a more or less probabilistic approach is always adopted: either via stochastic control techniques or (as it is in \cite{GLR15}) by strongly relying on Girsanov's theorem. On the contrary, we propose here a purely analytic proof which fits to the $\RCD$ framework, thus extending the previous results and including, in particular, the relevant case of Riemannian manifolds; as a further advantage, with slight modifications the same argument allows us to obtain the Hamilton-Jacobi-Bellman duality \eqref{eq:inthj} and, as a direct corollary, the Kantorovich-type duality formula \eqref{eq:intks} for the entropic cost, that were also missing in the Riemannian setting. This will be achieved as follows:
\begin{itemize}
\item[-] \eqref{eq:intbbs2} and \eqref{eq:intbbs} are proved in Theorems \ref{thm:main1} and \ref{thm:main1bis} respectively;
\item[-] \eqref{eq:inthj} is established in Theorem \ref{thm:duality};
\item[-] \eqref{eq:intks} is shown to hold in Theorem \ref{thm:main2}.
\end{itemize}
We thus provide a complete and unifying picture of the equivalent variational representations of the Schr\"odinger problem as well as a perfect parallelism with the analogous formulas for the Wasserstein distance. If we replace $\hr_{1/2}$ by $\hr_{\eps/2}$ in \eqref{eq:schrodinger}, denote by $\mathscr{I}_\eps(\mu_0,\mu_1)$ the minimal value of the associated problem and rescale properly Fokker-Planck and Hamilton-Jacobi-Bellman equations, this can be summarized as follows

\medskip

\begin{center}
\small
\begin{tabular}{|c|c|c|}
\hline \rule{0pt}{2.8ex}
$\,$ & Optimal transport & Schr\"odinger problem \\ 
\rule[-1.4ex]{0pt}{0pt} $\,$ & $W_2^2(\mu_0,\mu_1)/2$ & $\eps\mathscr{I}_\eps(\mu_0,\mu_1)$ \\ 
\hline \rule{0pt}{4.5ex} 
Primal problem & \multirow{2}{*}{$\displaystyle{\inf_{\sggamma \in \adm(\mu_0,\mu_1)} \int\frac{|x-y|^2}{2}\d\ggamma(x,y)}$} & \multirow{2}{*}{$\displaystyle{\inf_{\sggamma \in \adm(\mu_0,\mu_1)} \eps H(\ggamma\,|\,\hR)}$} \\ \rule[-2.8ex]{0pt}{0pt} Static version & $\,$ & $\,$ \\ 
\hline \rule{0pt}{4.8ex} \rule[-2.8ex]{0pt}{0pt}
$\,$ & \multirow{4}{*}[-2em]{$\displaystyle{\inf_{CE}\iint_0^1 \frac{|v_t|^2}{2}\rho_t\,\dt\d\mathcal{L}^d}$} & $\displaystyle{\eps H_0 + \inf_{fFP}\int_0^1\int\frac{|v^+_t|^2}{2}\d\nu^+_t\dt}$ \\ Primal problem & $\,$ & $\displaystyle{\eps H_1 + \inf_{bFP}\int_0^1\int\frac{|v^-_t|^2}{2}\d\nu^-_t\dt}$ \\ Dynamical version & $\,$ & $\displaystyle{\frac{\eps}{2}\big(H_0 + H_1\big) + \inf_{CE}\iint_0^1\Big( \frac{|v_t|^2}{2}}$ \\ \rule[-2.8ex]{0pt}{0pt} $\,$ & & $\displaystyle{+ \frac{1}{8}|\nabla\log\eta_t|^2 \Big)\eta_t\dt\d\mathcal{L}^d}$ \\ 
\hline 
Dual problem & \multirow{2}{*}[-0.8em]{$\displaystyle{\sup_{\phi \in C_b(\R^d)}\int Q_1\phi\,\d\mu_1 - \int\phi\,\d\mu_0}$} & \rule{0pt}{4ex} $\displaystyle{\eps H_0 + \sup_{\phi \in C_b(\R^d)}\int\phi\,\d\mu_1 - \int \tilde{Q}_1\phi\,\d\mu_0}$ \\ Static version & & \rule[-3.5ex]{0pt}{0pt} $\displaystyle{\eps H_1 + \sup_{\phi \in C_b(\R^d)}\int\phi\,\d\mu_0 - \int \tilde{Q}_1\phi\,\d\mu_1}$ \\
\hline 
Dual problem & \multirow{2}{*}[-0.8em]{$\displaystyle{\sup_{HJ}\int\phi(\cdot,1)\,\d\mu_1 - \int\phi(\cdot,0)\,\d\mu_0}$} & \rule{0pt}{4ex} $\displaystyle{\eps H_0 + \sup_{bHJB}\int \phi(\cdot,1)\,\d\mu_1 - \int \phi(\cdot,0)\,\d\mu_0 }$ \\ Dynamical version & & \rule[-3.5ex]{0pt}{0pt} $\displaystyle{\eps H_1 + \sup_{fHJB}\int \tilde{\phi}(\cdot,0)\,\d\mu_0 - \int \tilde{\phi}(\cdot,1)\,\d\mu_1 }$ \\
\hline 
\end{tabular}
\end{center}

\medskip

\noindent where $H_0 := H(\mu_0\,|\,\mathcal{L}^d)$, $H_1 := H(\mu_1\,|\,\mathcal{L}^d)$ and CE, FP, HJ, HJB, f and b are short-hand notations for continuity equation, Fokker-Planck, Hamilton-Jacobi, Hamilton-Jacobi-Bellman, forward and backward respectively.

\bigskip

The choice of replacing $\hr_{1/2}$ with $\hr_{\eps/2}$ in \eqref{eq:schrodinger} is motivated by the fact that optimal transport and Schr\"odinger problem are intertwined by an even stronger link, as one can guess from letting $\eps \downarrow 0$ in the right-hand column above. Indeed, the Monge-Kantorovich probem can be seen as the zero-noise limit of rescaled Schr\"odinger problems. The basic idea is that if the heat kernel admits the asymptotic expansion $\eps\log \hr_\eps(x,y)\sim -\frac{\sfd^2(x,y)}{2}$ (in the sense of Large Deviations), then the rescaled entropy functionals $\eps H(\cdot\,|\, \hR_\eps)$ converge to $\frac12\int \sfd^2(x,y)\,\d \cdot $ (in the sense of $\Gamma$-convergence). This has been obtained by Mikami in \cite{Mikami04} for the quadratic cost on $\R^d$, later on by Mikami-Thieullen \cite{MikamiThieullen08} for more general cost functions and finally by L\'eonard \cite{Leonard12} for Polish spaces and general diffusion processes (we refer to \cite{Leonard14} for a deeper discussion of this topic, historical remarks and much more). For this reason and to highlight the rescaling factor, throughout the paper we shall always make $\eps$ explicit.

\section{Preliminaries}

\subsection{Analysis and optimal transport in \texorpdfstring{$\RCD$}{RCD} spaces}

By $C([0,1],(\X,\sfd))$, or simply $C([0,1],\X)$, we denote the space of continuous curves with values on the metric space $(\X,\sfd)$. For the notion of {\bf absolutely continuous curve} in a metric space and of {\bf metric speed} see for instance Section 1.1 in \cite{AmbrosioGigliSavare08}. The collection of absolutely continuous curves on $[0,1]$ is denoted $AC([0,1],(\X,\sfd))$, or simply by $AC([0,1],\X)$.

By $\prob\X$ we denote the space of Borel probability measures on $(\X,\sfd)$ and by $\probt\X \subset \prob\X$ the subclass of those with finite second moment.

\medskip

Let $(\X,\sfd,\mm)$ be a complete and separable metric measure space endowed with a Borel non-negative measure which is finite on bounded sets. 

For the definition of the {\bf Sobolev space} $W^{1,2}(\X)$ and of {\bf minimal weak upper gradient} $|D f|$ see \cite{Cheeger00} (and the works \cite{AmbrosioGigliSavare11}, \cite{Shanmugalingam00} for alternative - but equivalent - definitions of Sobolev functions). The local counterpart of $W^{1,2}(\X)$ is introduced as follows: $L^2_{loc}(\X)$ is defined as the space of functions $f \in L^0(\X)$ such that for all open set $\Omega \subset \X$ with compact closure there exists a function $g \in L^2(\X)$ such that $f = g$ $\mm$-a.e.\ in $\Omega$ and the local Sobolev space $W^{1,2}_{loc}(\X)$ is then defined as
\begin{equation}
\label{eq:3}
W^{1,2}_{loc}(\X) := \{ f \in L^0(\X) \,:\, \forall \Omega \subset\subset \X \,\,\exists g \in W^{1,2}(\X) \textrm{ s.t. } f = g \,\,\mm\textrm{-a.e. in } \Omega\}.
\end{equation}
The local minimal weak upper gradient of a function $f \in W^{1,2}_{loc}(\X)$ is denoted by $|Df|$, omitting the locality feature, and defined for all $\Omega \subset\subset \X$ as $|Df| := |Dg|$ $\mm$-a.e.\ in $\Omega$, where $g$ is as in (\ref{eq:3}). The definition does depend neither on $\Omega$ nor on the choice of $g$ associated to it by locality of the minimal weak upper gradient.

If $W^{1,2}(\X)$ is Hilbert, which from now on we shall always assume, then $(\X,\sfd,\mm)$ is said {\bf infinitesimally Hilbertian} (see \cite{Gigli12}). The language of $L^0$-normed modules (see \cite{Gigli14}) allows to introduce the {\bf differential} as a well-defined linear map $\d$ from $W^{1,2}_{loc}(\X)$ with values in $L^0(T^*\X)$, the family of (measurable) 1-forms. The dual of $L^0(T^*\X)$ as $L^0$-normed module is denoted by $L^0(T\X)$, it is canonically isomorphic to $L^0(T^*\X)$ and its elements are called vector fields; the isomorphism sends the differential $\d f$ to the gradient $\nabla f$.

After $W^{1,2}_{loc}(\X)$ we can also introduce
\[
\begin{split}
D({\rm div}_{loc}) & := \{ v \in L^0(T\X) \,:\, \forall \Omega \subset\subset \X \,\,\exists w \in D({\rm div}) \textrm{ s.t. } v = w \,\,\mm\textrm{-a.e. in } \Omega\} \\
D(\Delta_{loc}) & := \{ f \in L^0(\X) \,:\, \forall \Omega \subset\subset \X \,\,\exists g \in D(\Delta) \textrm{ s.t. } f = g \,\,\mm\textrm{-a.e. in } \Omega\}
\end{split}
\]
so that the notions of {\bf divergence} and {\bf Laplacian} can be extended by locality to locally integrable vector fields and functions respectively.

As regards the properties of $\d, {\rm div},\Delta$, the differential satisfies the following calculus rules which we shall use extensively without further notice:
\begin{align*}
|\d f|&=|D f|\quad\mm\ae&&\forall f\in S^2(\X)\\
\d f&=\d g\qquad\mm\ae\ \text{\rm on}\ \{f=g\} &&\forall f,g\in S^2(\X)\\
\d(\varphi\circ f)&=\varphi'\circ f\,\d f&&\forall f\in S^2(\X),\ \varphi:\R\to \R\ \text{Lipschitz}\\
\d(fg)&=g\,\d f+f\,\d g&&\forall f,g\in L^\infty\cap S^2(\X)
\end{align*}
where it is part of the properties the fact that $\varphi\circ f,fg\in S^2(\X)$ for $\varphi,f,g$ as above. For the divergence, the formula
\[
{\rm div}(fv)=\d f(v)+f{\rm div}(v)\qquad\forall f\in W^{1,2}(\X),\ v\in D({\rm div}),\ \text{such that}\ |f|,|v|\in L^\infty(\X)
\]
holds, where it is intended in particular that $fv\in D({\rm div})$ for $f,v$ as above, and for the Laplacian
\[
\begin{split}
\Delta(\varphi\circ f)&=\varphi''\circ f|\d f|^2+\varphi'\circ f\Delta f\\
\Delta(fg)&=g\Delta f+f\Delta g+2\la\nabla f,\nabla g\ra
\end{split}
\]
where in the first equality we assume that $f\in D(\Delta),\varphi\in C^2(\R)$ are such that $f,|\d f|\in L^\infty(\X)$ and $\varphi',\varphi''\in L^\infty(\R)$ and in the second that $f,g\in D(\Delta)\cap L^\infty(\X)$ and $|\d f|,|\d g|\in L^\infty(\X)$ and it is part of the claims that $\varphi\circ f,fg$ are in $D(\Delta)$. On $W^{1,2}_{loc}(\X)$ as well as on $D({\rm div}_{loc})$ and $D(\Delta_{loc})$ the same calculus rules hold with slight adaptations (see for instance \cite{GigTam18}).

The Laplacian $\Delta$ is the infinitesimal generator of a 1-parameter semigroup $(\h_t)$ called {\bf heat flow} (see \cite{AmbrosioGigliSavare11}). For such a flow it holds
\begin{equation}\label{eq:heatreg}
u \in L^2(\X) \qquad \Rightarrow \qquad (\h_t u) \in C([0,\infty),L^2(\X)) \cap AC_{loc}((0,\infty),W^{1,2}(\X))
\end{equation}
and for any $u \in L^2(\X)$ the curve $t \mapsto \h_t u$ is the only solution of
\[
\ddt\h_t u = \Delta\h_t u \qquad \h_t u \to u\text{ as }t\downarrow0.
\]
If moreover $(\X,\sfd,\mm)$ is an $\RCD(K,\infty)$ space (see \cite{AmbrosioGigliSavare11-2}), the following a priori estimates hold true for every $u \in L^2(\X)$ and $t > 0$:
\begin{equation}\label{eq:apriori}
\| |\nabla\h_t u| \|^2_{L^2(\X)} \leq \frac{1}{2t}\|u\|^2_{L^2(\X)} \qquad\qquad \|\Delta\h_t u\|^2_{L^2(\X)} \leq \frac{1}{2t^2} \|u\|^2_{L^2(\X)}
\end{equation}
and the {\bf Bakry-\'Emery contraction estimate} (see \cite{AmbrosioGigliSavare11-2}) is satisfied:
\begin{equation}
\label{eq:be}
|\d\h_tf|^2\leq e^{-2Kt}\h_t(|\d f|^2) \quad \mm\ae \qquad\forall f\in W^{1,2}(\X),\ t\geq 0.
\end{equation}
Furthermore if $u \in L^\infty(\X)$, then $\h_t u$ is Lipschitz on $\supp(\mm)$ for all $t>0$ and
\begin{equation}\label{eq:lipreg}
\sqrt{2\int_0^t e^{2Ks}\d s}\Lip(\h_t u) \leq \|u\|_{L^\infty(\X)}.
\end{equation}
Still within the $\RCD$ framework, there exists the {\bf heat kernel}, namely a function 
\begin{equation}
\label{eq:hk}
(0,\infty)\times \X^2\ni (t,x,y)\quad\mapsto\quad \hr_t[x](y)=\hr_t[y](x)\in (0,\infty)
\end{equation}
such that $\h_tf(x)=\int f(y)\hr_t[x](y)\,\d\mm(y)$ for all $t>0$ and for every $f \in L^2(\X)$. For every $x\in \X$ and $t>0$, $\hr_t[x]$ is a probability density; thus the heat flow can be extended to $L^1(\X)$, is {\bf mass preserving} and satisfies the {\bf maximum principle}, i.e.
\[
f\leq c\quad\mm-a.e.\qquad\qquad\Rightarrow \qquad\qquad\h_tf\leq c\quad\mm\ae,\ \forall t>0.
\]
On finite-dimensional $\RCD^*(K,N)$ spaces (introduced in \cite{Gigli12}), a well-known consequence of lower Ricci curvature bounds (see e.g. \cite{Cheeger-Colding97I}, \cite{Cheeger-Colding97II}, \cite{Cheeger-Colding97III}) is the existence of `{\bf good cut-off functions}', typically intended as cut-offs with bounded Laplacian; for our purposes the following result will be sufficient:

\begin{Lemma}\label{lem:cutoff}
Let $(\X,\sfd,\mm)$ be an $\RCD^*(K,N)$ space with $K \in \R$ and $N \in [1,\infty)$. Then for all $R>0$ and $x \in \X$ there exists a function $\nchi_R : \X \to \R$ satisfying:
\begin{itemize}
\item[(i)] $0 \leq \nchi_R \leq 1$, $\nchi_R \equiv 1$ on $B_R(x)$ and $\supp(\nchi_R) \subset B_{R+1}(x)$;
\item[(ii)] $\nchi_R \in D(\Delta) \cap L^\infty(\X)$, $|\nabla\nchi_R| \in L^\infty(\X)$, $\Delta\nchi_R \in W^{1,2} \cap L^\infty(\X)$.
\end{itemize}
Moreover, there exist constants $C,C' > 0$ depending on $K,N$ only such that
\begin{equation}\label{eq:cutbound}
\| |\nabla\nchi_R| \|_{L^{\infty}(\X)} \leq C, \qquad\qquad \|\Delta\nchi_R\|_{L^{\infty}(\X)} \leq C'.
\end{equation}
\end{Lemma}

The proof can be obtained following verbatim the arguments given in Lemma 3.1 of \cite{Mondino-Naber14} (inspired by  \cite{AmbrosioMondinoSavare13-2}, see also  \cite{Gigli-Mosconi14} for an alternative approach): there the authors are interested in cut-off functions such that $\nchi \equiv 1$ on $B_R(x)$ and $\supp(\nchi) \subset B_{2R}(x)$: for this reason they fix $R>0$ and then claim that for all $x \in \X$ and $0<r<R$ there exists a cut-off function $\nchi$ satisfying $(i)$, $(ii)$ and \eqref{eq:cutbound} with $C,C'$ also depending on $R$. However, as far as one is concerned with cut-off functions $\nchi$ where the distance between $\{\nchi = 0\}$ and $\{\nchi = 1\}$ is always equal to 1, the proof of \cite{Mondino-Naber14} in the case $R=1$ applies and does not affect \eqref{eq:cutbound}.

\medskip

We conclude recalling that on $\RCD^*(K,N)$ spaces with $N \in [1,\infty)$ the reference measure $\mm$ satisfies the following volume growth condition: there exists a constant $C>0$ such that
\begin{equation}\label{eq:volgrowth}
\mm(B_r(x)) \leq Ce^{Cr}, \qquad \forall x \in \X,\, r>0.
\end{equation}
For this reason we shall consider the weighted $L^2(\X,e^{-V}\mm)$ and $W^{1,2}(\X,e^{-V}\mm)$ spaces, where $V := M\sfd^2(\cdot,\bar{x})$. Indeed $e^{-V}\mm$ has finite mass for every $M>0$. For $L^2(\X,e^{-V}\mm)$ no comments are required. The weighted Sobolev space is defined as
\[
W^{1,2}(\X,e^{-V}\mm) := \{ f \in W^{1,2}_{loc}(\X) \,:\, f,|Df| \in L^2(\X,e^{-V}\mm) \}
\]
where $|Df|$ is the local minimal weak upper gradient already introduced. Since $V$ is locally bounded, $W^{1,2}(\X,e^{-V}\mm)$ turns out to coincide with the Sobolev space built over the metric measure space $(\X,\sfd,e^{-V}\mm)$, thus motivating the choice of the notation.

\subsection{The Schr\"odinger problem in \texorpdfstring{$\RCD$}{RCD} spaces}\label{sec:2.2}

Let us first recall the definition of the relative entropy functional in the case of a reference measure with possibly infinite mass (see \cite{Leonard14b} for more details). Given a $\sigma$-finite measure $\nu$ on a Polish space $(\Y,\tau)$, there exists a measurable function $W : \Y \to [0,\infty)$ such that
\[
z_W := \int e^{-W}\d\nu < +\infty.
\]
Introducing the probability measure $\nu_W := z_W^{-1}e^{-W}\nu$, for any $\sigma \in \prob\Y$ such that $\int W\d\sigma < +\infty$ the {\bf Boltzmann-Shannon entropy} is defined as
\begin{equation}\label{eq:entdef}
H(\sigma\,|\,\nu) := H(\sigma\,|\,\nu_W) - \int W\d\sigma - \log z_W
\end{equation}
where $H(\sigma\,|\,\nu_W) := \int\rho\log(\rho)\,\d\nu_W$ if $\sigma = \rho\nu_W$ and $+\infty$ otherwise; notice that Jensen's inequality and the fact that $\tilde{\nu} \in \prob\Y$ grant that $\int\rho\log(\rho)\,\d\tilde{\nu}$ is well-defined and non-negative, in particular the definition makes sense.

Because of \eqref{eq:volgrowth}, on an $\RCD^*(K,N)$ space $(\X,\sfd,\mm)$ with $K \in \R$ and $N \in [1,\infty)$ we can choose $W = \sfd^2(\cdot,\bar{x})$ in the definition above, so that $H(\cdot\,|\,\mm)$ turns out to be well-defined on $\probt\X$ and $W_2$-lower semicontinuous. If we also introduce the following measure on $\X^2$
\[
\d \hR^{\eps}(x,y):= \hr_{\eps}[x](y)\,\d\mm(x)\,\d\mm(y),
\]
where $\hr_\eps[x](y)$ is the heat kernel \eqref{eq:hk}, then the choice $W : \X^2 \to [0,\infty)$, $W(x,x') := \sfd^2(x,\bar{x}) + \sfd^2(x',\bar{x})$ entails that, given any two probability measures $\mu_0=\rho_0\mm$, $\mu_1=\rho_1\mm$ with bounded densities and supports, $H(\cdot\,|\,\hR^\eps)$ is well-defined in $\adm(\mu_0,\mu_1)$ and narrowly lower semicontinuous therein, as shown in \cite{GigTam18}.

Therefore, the minimization problem
\[
\inf_{\sggamma \in \adm(\mu_0,\mu_1)} H(\ggamma\,|\,\hR^{\eps/2}),
\]
also known as {\bf Schr\"odinger problem} (the choice of working with $\hR^{\eps/2}$ is convenient for the computations we will do later on) is meaningful. Actually, given $\mu_0,\mu_1$ as above, there exists a unique minimizer $\ggamma^\eps$ and $\ggamma^\eps = f^\eps\otimes g^\eps\hR^{\eps/2}$ for appropriate Borel functions $f,g : \X \to [0,\infty)$ which are $\mm$-a.e.\ unique up to the trivial transformation $(f,g)\to (cf,g/c)$ for some $c>0$. In addition, $f^\eps,g^\eps$ belong to $L^{\infty}(\X)$ and their supports are included in $\supp(\mu_0)$ and $\supp(\mu_1)$ respectively (cf.\ Proposition 2.1 and Theorem 2.2 in \cite{GigTam18}). Thus, the {\bf entropic cost} $\mathscr{I}_\eps$ relative to $\hR^{\eps/2}$, defined as
\[
\mathscr{I}_\eps(\mu_0,\mu_1) := \min_{\sggamma \in \adm(\mu_0,\mu_1)} H(\ggamma\,|\,\hR^{\eps/2}),
\]
is finite.
 
Now let us fix the notations that we shall use in the sequel. For any $\eps>0$ we set $\rho^\eps_0:=\rho_0$, $\rho^\eps_1:=\rho_1$, $\mu^\eps_0:=\mu_0$, $\mu^\eps_1:=\mu_1$ and 
\[
\left\{\begin{array}{l}
f^{\varepsilon}_t := \h_{\varepsilon t/2}f^{\varepsilon} \\
\\
\varphi_t^{\varepsilon} := \varepsilon\log f_t^{\varepsilon}\\
\\
\text{for }t\in(0,1]
\end{array}
\right.\qquad\qquad
\left\{\begin{array}{l}
g^{\varepsilon}_t := \h_{\varepsilon(1-t)/2}g^{\varepsilon} \\
\\
\psi_t^{\varepsilon} := \varepsilon\log g_t^{\varepsilon}\\
\\
\text{for }t\in[0,1)
\end{array}
\right.\qquad\qquad
\left\{\begin{array}{l}
\rho^{\varepsilon}_t := f^{\varepsilon}_t g^{\varepsilon}_t \\
\\
\mu^{\varepsilon}_t := \rho^{\varepsilon}_t\mm\\
\\
\vartheta^{\varepsilon}_t := \frac12({\psi^{\varepsilon}_t - \varphi^{\varepsilon}_t})\\
\\
\text{for }t\in(0,1)
\end{array}
\right.
\]
and we also define
\begin{equation}\label{eq:def01}
\begin{split}
& \varphi_0^\eps := \eps\log(f^\eps) \qquad \textrm{in }\supp(\mu_0), \\
& \psi_1^\eps := \eps\log(g^\eps) \qquad \textrm{in }\supp(\mu_1).
\end{split}
\end{equation}
As shown in \cite{GigTam18} all the functions above are well defined, $\mu^\eps_t \in \probt\X$ for every $t\in[0,1],\eps > 0$
\begin{equation}\label{eq:ent-continuity}
[0,1] \ni t \mapsto H(\mu_t^\eps\,|\,\mm) \quad \textrm{is continuous}
\end{equation}
and moreover for any $\eps > 0$ it holds:
\begin{itemize}
\item[a)] $f^\eps_t,g^\eps_t,\rho^\eps_t$ belong to $D(\Delta)$ for all $t \in \mathcal{I}$, where $\mathcal{I}$ is the respective domain of definition (for $(\rho^\eps_t)$ we pick $\mathcal{I} = (0,1)$);
\item[b)] $\varphi^\eps_t,\psi^\eps_t,\vartheta^\eps_t$ belong to $D(\Delta_{loc})$ for all $t \in \mathcal{I}$, where $\mathcal{I}$ is the respective domain of definition.
\end{itemize}
Secondly, $(f^\eps_t),(g^\eps_t),(\rho^\eps_t) \in C([0,1],L^2(\X)) \cap AC_{loc}(\mathcal{I},W^{1,2}(\X)) \cap L^\infty([0,1],L^\infty(\X))$ for any $\eps > 0$ and their time derivatives are given by the following expressions for a.e.\ $t\in[0,1]$:
\begin{equation}\label{eq:pde1}
\ddt f^\eps_t = \frac{\eps}{2}\Delta f^\eps_t \qquad \ddt g^\eps_t = -\frac{\eps}{2}\Delta g^\eps_t \qquad \ddt\rho^\eps_t + {\rm div}(\rho^\eps_t\nabla\vartheta^\eps_t) = 0.
\end{equation}
As concerns $(\varphi_t^\eps),(\psi_t^\eps),(\vartheta_t^\eps)$, for all $\mathcal{C} \subset \mathcal{I}$ compact and $\bar{x} \in \X$ there exists $M > 0$ depending on $K,N,\rho_0,\rho_1,\mathcal{C},\bar{x}$ such that they belong to $AC(\mathcal{C},W^{1,2}(\X,e^{-V}\mm))$ where $V = M\sfd^2(\cdot,\bar{x})$; their time derivatives are given by the following expressions for a.e.\ $t\in[0,1]$:
\begin{equation}\label{eq:pde2}
\begin{split}
\ddt\varphi^\eps_t & = \frac{1}{2}|\nabla\varphi^\eps_t|^2 + \frac{\eps}{2}\Delta\varphi^\eps_t \qquad -\ddt\psi^\eps_t = \frac{1}{2}|\nabla\psi^\eps_t|^2 + \frac{\eps}{2}\Delta\psi^\eps_t \\
& \ddt\vartheta^\eps_t + \frac{|\nabla\vartheta^\eps_t|^2}{2} = -\frac{\eps^2}{8}\Big(2\Delta\log\rho^\eps_t + |\nabla\log\rho^\eps_t|^2\Big).
\end{split}
\end{equation}
In addition, for every $\delta \in (0,1)$ and $\bar{x} \in \X$ there exist constants $C,C' > 0$ which depend on $K,N,\bar{x},\rho_0,\rho_1$ and $C''>0$ (depending also on $\delta$) such that
\begin{subequations}
\begin{align}
\label{eq:tail}
\rho_t^\eps & \leq C e^{-C'\sfd^2(\cdot,\bar{x})}, \qquad \mm\ae,\forall t \in [0,1], \\
\label{eq:lipcontr}
\lip(\varphi^\eps_t) + \lip(\psi^\eps_{1-t}) & \leq C''\big(1 + \sfd(\cdot,\bar{x})\big), \qquad \mm\ae,\forall t \in [\delta,1].
\end{align}
\end{subequations}
As a final remark, let us recall (Lemma 4.9 in \cite{GigTam18}) that
\begin{equation}\label{eq:finitekin}
\iint_0^1 |\nabla\varphi_t^\eps|^2\rho_t^\eps\dt\d\mm < \infty \qquad \iint_0^1 |\nabla\psi_t^\eps|^2\rho_t^\eps\dt\d\mm < \infty \qquad \iint_0^1 |\nabla\vartheta_t^\eps|^2\rho_t^\eps\dt\d\mm < \infty.
\end{equation}

\section{Lemmas and a first dynamical viewpoint on \texorpdfstring{$\mathscr{I}_\eps$}{I}}

\subsection{Solutions of continuity, Fokker-Planck, Hamilton-Jacobi-Bellman equations: definition}

We start giving the definition of `distributional' solutions of the continuity equation and of the forward/backward Fokker-Planck equation in our setting:

\begin{Definition}\label{def:fpe}
Let $(\X,\sfd,\mm)$ be an infinitesimally Hilbertian metric measure space, $t \mapsto X_t \in L^0(T\X)$ a Borel family of vector fields, possibly defined only for a.e.\ $t \in [0,1]$, and $c \geq 0$, $\sigma \in \{-1,1\}$. A curve $(\mu_t)\subset \probt\X$ is a solution of
\[
\sigma\ddt\mu_t + {\rm div}(X_t\mu_t) = c\Delta\mu_t
\]
if:
\begin{itemize}
\item[(i)] it is weakly continuous and there exists $C>0$ such that $\mu_t \leq C\mm$ for all $t \in [0,1]$;
\item[(ii)] the map $t \mapsto \int |X_t|^2\d\mu_t$ is Borel and belongs to $L^1(0,1)$;
\item[(iii)] for any $f \in D(\Delta)$ the map $[0,1] \ni t\mapsto\int f\,\d\mu_t$ is absolutely continuous and it holds
\[
\sigma\ddt\int f\,\d\mu_t = \int \Big(\d f(X_t) + c\Delta f\Big)\d\mu_t\qquad {\rm a.e.}\ t.
\]
\end{itemize}
When
\begin{itemize}
\item[-] $\sigma = 1$ and $c>0$, $(\mu_t)$ is said to be a solution of the forward Fokker-Planck equation;
\item[-] $\sigma = -1$ and $c>0$, $(\mu_t)$ is said to be a solution of the backward Fokker-Planck equation;
\item[-] $c = 0$, $(\mu_t)$ is said to be a solution of the continuity equation.
\end{itemize}
We will refer to $(X_t)$ as drift or velocity field.
\end{Definition}

Let us point out that this definition of solution of the continuity equation is consistent with the one proposed in \cite{GigliHan13} and recalled in the Introduction, because if
\[
\ddt\int f\,\d\mu_t = \int \d f(X_t)\,\d\mu_t\qquad {\rm a.e.}\ t
\]
holds for every $f \in D(\Delta)$, then it also holds for $f \in W^{1,2}(\X)$: it is sufficient to integrate the equality on $[t_0,t_1] \subset [0,1]$ and argue by density thanks to the fact that by $(ii)$
\[
\int_0^1\int |X_t|^2\d\mu_t\dt < \infty.
\]
In view of Theorem \ref{thm:duality} let us also provide a suitable notion of `strong' supersolution of the forward/backward Fokker-Planck equation.

\begin{Definition}\label{def:hjb}
Let $\sigma \in \{-1,1\}$ and $c,T>0$. A curve $[0,T] \ni t \mapsto \phi_t \in L^0(\X)$ is a supersolution of the forward (resp.\ backward) Hamilton-Jacobi-Bellman equation provided:
\begin{itemize}
\item[(i)] there exists $C>0$ such that $\|\phi_t\|_{L^\infty(\X)} \leq C$ for all $t \in [0,T]$;
\item[(ii)] $\phi_t \in D(\Delta_{loc})$ for all $t \in [0,T]$ with $(|\nabla\phi_t|),(\Delta\phi_t) \in L^\infty((0,T),L^2(\X))$.
\item[(iii)] there exist $x \in \X$ and $M>0$ such that $(\phi_t) \in AC([0,T],L^2(\X,e^{-V}\mm))$, where $V := M\sfd^2(\cdot,x)$, and its time derivative satisfies
\[
\sigma\ddt\phi_t \geq \frac{1}{2}|\nabla\phi_t|^2 + c\Delta\phi_t \qquad \textrm{for a.e. } t \in [0,T]
\]
with $\sigma = 1$ (resp.\ $\sigma = -1$).
\end{itemize}
\end{Definition}

\subsection{Technical lemmas}

In this section we collect some auxiliary results that will be used several times in the proof of the main theorems. We start with an integrability statement (a stronger result is actually true, see \cite{GigTam18}, but this is sufficient for our purposes).

\begin{Lemma}\label{lem:continuity}
With the same assumptions and notation as in Section \ref{sec:2.2}, the following holds.

For any $\eps > 0$ and $t \in \mathcal{I}$ let $h_t^\eps$ denote any of $\varphi_t^\eps,\psi_t^\eps,\vartheta_t^\eps,\log\rho_t^\eps$ and let $H_t^\eps$ denote any of the functions
\[
\rho_t^\eps h_t^\eps, \quad \rho_t^\eps |h_t^\eps|^2, \quad \rho_t^\eps |\nabla h_t^\eps|^2,
\]
where $\mathcal{I}$ is the domain of definition of $h_t^\eps$ (for $\log\rho_t^\eps$ we pick $\mathcal{I} = (0,1)$). Then $H_t^\eps \in L^1(\X)$ for every $\eps \in (0,1)$ and $t \in \mathcal{I}$ and, for any $\mathcal{C} \subset\subset \mathcal{I}$, $H_{\cdot}^\eps \in L^1(\mathcal{C} \times \X,\d t \otimes \mm)$. Moreover, $\mathcal{I} \ni t \mapsto \int H_t^\eps\,\d\mm$ is continuous.
\end{Lemma}

\begin{proof}
Since $(\rho_t^\eps) \in C([0,1],L^2(\X))$ and $(h_t^\eps),(|\nabla h_t^\eps) \in C(\mathcal{I},L^2(\X,e^{-V}\mm))$, all the functions appearing in the statement are continuous from $\mathcal{I}$ to $L^0(\X)$ equipped with the topology of convergence in measure on bounded sets. Therefore the continuity of $\mathcal{I} \ni t \mapsto \int H_t^\eps\,\d\mm$ for these maps will follow as soon as we show that they are, locally in $t \in \mathcal{I}$, uniformly dominated by an $L^1(\X)$ function. This will also imply all the other statements. Furthermore, it is sufficient to consider the case $h_t^\eps = \varphi_t^\eps$, as the estimates for $\psi_t^\eps$ can be obtained by symmetric arguments and the ones for $\vartheta_t^\eps,\log\rho_t^\eps$ follow from the identities $\vartheta_t^\eps = \frac{\psi_t^\eps - \varphi_t^\eps}{2}$ and $\eps\log\rho_t^\eps = \varphi_t^\eps + \psi_t^\eps$.

From \eqref{eq:tail} and \eqref{eq:lipcontr} we immediately see that for any $\bar{x} \in \X$ and $\delta > 0$ there exist constants $c_1,c_2 > 0$ depending on $K,N,\delta,\bar{x},\rho_0,\rho_1$ only such that
\[
\rho_t^\eps|\nabla\varphi_t^\eps|^2 \leq c_1\big(1+\sfd^2(\cdot,\bar{x})\big)\exp\big(-c_2\sfd^2(\cdot,\bar{x})\big) \qquad \mm\ae
\]
for every $t \in [\delta,1]$ and $\eps \in (0,1)$. The volume growth \eqref{eq:volgrowth} then implies that the right-hand side is integrable and thus the conclusion.

For $\rho_t^\eps \varphi_t^\eps$ and $\rho_t^\eps |h_t^\eps|^2$, observe that from \eqref{eq:lipcontr} and the fact that $\X$ is a geodesic space it follows that
\[
|\varphi_t^\eps(x) - \varphi_t^\eps(\bar{x})| \leq C_{\delta}\sfd(x,\bar{x})(1+\sfd(x,\bar{x})) \leq C_{\delta}(1+\sfd^2(x,\bar{x})) \quad \forall t \in [\delta,1],
\]
which means that $\varphi_t^\eps$ has quadratic growth. We then argue as before, coupling this information with \eqref{eq:tail} and \eqref{eq:volgrowth}.
\end{proof}

The following result is in the same spirit of the previous lemma and of the reminders of Section \ref{sec:2.2}.

\begin{Lemma}\label{lem:hjb}
Let $(\X,\sfd,\mm)$ be an $\RCD^*(K,N)$ space with $K \in \R$ and $N < \infty$, $u \in L^2 \cap L^\infty(\X)$ be non-negative and $\delta > 0$. Put $\phi_t^\delta := \log(\h_t u + \delta)$ for all $t \geq 0$. Then:
\begin{itemize}
\item[(i)] there exists $C>0$ such that $\|\phi_t^\delta\|_{L^\infty(\X)} \leq C$ for all $t \geq 0$;
\item[(ii)] for all $x \in \X$ and $M>0$, $(\phi_t^\delta) \in C([0,\infty),L^2(\X,e^{-V}\mm)) \cap AC_{loc}((0,\infty),L^2(\X,e^{-V}\mm))$, where $V := M\sfd^2(\cdot,x)$, and its time derivative is given by
\begin{equation}\label{eq:hjb-eta}
\ddt\phi_t^\delta = |\nabla\phi_t^\delta|^2 + \Delta\phi_t^\delta \qquad \textrm{for a.e. } t>0;
\end{equation}
\item[(iii)] $(|\nabla\phi_t^\delta|),(\Delta\phi_t^\delta) \in L^\infty_{loc}((0,\infty),L^2(\X))$;
\item[(iv)] let $(\mu_t)_{t \geq 0} \subset \prob\X$ be weakly continuous with $\mu_t \leq C\mm$ for some $C>0$ independent of $t$, set $\eta_t := \frac{\d\mu_t}{\d\mm}$ and denote by $H_t^\delta$ any of the functions
\[
\phi_t^\delta\eta_t, \quad |\phi_t^\delta|^2\eta_t, \quad |\nabla\phi_t^\delta|\eta_t, \quad |\nabla\phi_t^\delta|^2\eta_t.
\]
Then $H_t^\delta \in L^1(\X)$ for every $t,\delta > 0$ and, for any $\mathcal{C} \subset\subset (0,\infty)$, $H_{\cdot}^\delta \in L^1(\mathcal{C} \times \X,\dt \otimes \mm)$.
\end{itemize}
\end{Lemma}

\begin{proof}
Fix $x \in \X$, $M>0$ and let $V$ be defined as in the statement. By the maximum principle for the heat flow $\log\delta \leq \phi_t^\delta \leq \log(\|u\|_{L^\infty(\X)} + \delta)$ for all $t \geq 0$, so that
\[
\sup_{t \in [0,\infty)}\|\phi_t^\delta\|_{L^\infty(\X)} < \infty \qquad \textrm{and} \qquad (\phi_t^\delta) \in L^\infty((0,\infty),L^2(\X,e^{-V}\mm)).
\]
Since $\log$ is smooth with bounded derivatives on $[\delta,+\infty)$, the chain and Leibniz rules entail that
\[
|\nabla\phi_t^\delta| \leq \frac{|\nabla\h_t u|}{\delta} \qquad\qquad |\Delta\phi_t^\delta| \leq \frac{|\Delta\h_t u|}{\delta} + \frac{|\nabla\h_t u|^2}{\delta^2}
\]
and, by the a priori estimates \eqref{eq:apriori} and the Lipschitz regularization \eqref{eq:lipreg}, $(iii)$ follows. Furthermore, notice that \eqref{eq:heatreg} and the chain rule grant that $\mm$-a.e.\ \eqref{eq:hjb-eta} holds for a.e.\ $t$; since $(iii)$ implies in particular that $(|\nabla\phi_t^\delta|),(\Delta\phi_t^\delta) \in L^2_{loc}((0,\infty),L^2(\X,e^{-V}\mm))$, this means that $(\phi_t^\delta) \in AC_{loc}((0,\infty),L^2(\X,e^{-V}\mm))$ and \eqref{eq:hjb-eta} actually holds when the left-hand side is intended as limit of the difference quotients in $L^2(\X,e^{-V}\mm)$.

The continuity in $t=0$ follows by dominated convergence from $(i)$ and the fact that for any sequence $t_n \downarrow 0$ there exists a subsequence $t_{n_k} \downarrow 0$ such that $\h_{t_{n_k}}u \to u$ $\mm$-a.e.

Finally, given $\mathcal{C} \subset\subset (0,\infty)$, observe that from (i), (iii) and the fact that $\mu_t \leq C\mm$ for all $t \geq 0$ we get
\[
\sup_{t \in \mathcal{C}}\int |\phi_t^\delta|\eta_t\,\d\mm + \int |\phi_t^\delta|^2\eta_t\,\d\mm + \int |\nabla\phi_t^\delta|^2\eta_t\,\d\mm < \infty,
\]
whence integrability on $\mathcal{C} \times \X$ by Fubini's theorem. For $|\nabla\phi_t^\delta|\eta_t$ it is sufficient to notice that $|\nabla\phi_t^\delta|\eta_t \leq \frac{1}{2}|\nabla\phi_t^\delta|^2\eta_t + \frac{1}{2}\eta_t$ and then argue as above.
\end{proof}

We shall also make use of the following simple lemma valid on general metric measure spaces.

\begin{Lemma}\label{lem:dermiste}
Let $(\Y,\sfd_\Y,\mm_\Y)$ be an infinitesimally Hilbertian metric measure space endowed with a non-negative measure $\mm_\Y$ which is finite on bounded sets. Let $(\mu_t)$ be a solution of
\[
\sigma\ddt\mu_t + {\rm div}(X_t\mu_t) = c\Delta\mu_t
\]
in the sense of Definition \ref{def:fpe} and let $f \in D(\Delta)$. Then $t\mapsto\int f\,\d\mu_t$ is absolutely continuous and 
\begin{equation}
\label{eq:bder1}
\Big|\frac\d{\d t}\int f\,\d\mu_t \Big| \leq \int \big(|\d f||X_t| + c|\Delta f|\big)\d\mu_t\qquad{\rm a.e.}\ t\in[0,1],
\end{equation}
where the exceptional set can be chosen to be independent of $f$.

Moreover, if $(f_t)\in AC([0,1],L^2(\Y))\cap L^\infty([0,1],W^{1,2}(\Y))$ with $(\Delta f_t) \in L^\infty([0,1],L^2(\Y))$, then the map $t\mapsto\int f_t\,\d\mu_t$ is absolutely continuous and 
\begin{equation}\label{eq:derchain}
\frac\d{\d s}\Big(\int f_s\,\d\mu_s\Big)\restr{s=t}=\int \big(\frac\d{\d s}f_s\restr{s=t}\big)\,\d\mu_t + \frac\d{\d s}\Big(\int f_t\,\d\mu_s\Big)\restr{s=t}
\end{equation}
for a.e.\ $t \in [0,1]$.

If $c=0$, it is sufficient to assume $f \in W^{1,2}(\Y)$ in \eqref{eq:bder1} and $(f_t)\in AC([0,1],L^2(\Y))\cap L^\infty([0,1],W^{1,2}(\Y))$ in \eqref{eq:derchain}.
\end{Lemma}

\begin{proof}
The absolute continuity of $t \mapsto \int f\,\d\mu_t$ and the bound \eqref{eq:bder1} are trivial consequences of Definition \ref{def:fpe}. The fact that the exceptional set can be chosen independently of $f$ follows from the separability of $W^{1,2}(\Y)$ and standard approximation procedures, carried out, for instance, in \cite{Gigli14}.

This implies that the second derivative in the right hand side of \eqref{eq:derchain} exists for a.e.\ $t$, so that the claim makes sense. The absolute continuity of $t\mapsto\int f_t\,\d\mu_t$ follows from the fact that for any $t_0,t_1 \in [0,1]$, $t_0 < t_1$ it holds
\[
\begin{split}
\Big|\int f_{t_1}\,\d\mu_{t_1} - \int f_{t_0}\,\d\mu_{t_0}\Big| & \leq \Big|\int f_{t_1}\,\d\mu_{t_1} - \int f_{t_1}\,\d\mu_{t_0}\Big| + \Big|\int \big(f_{t_1} - f_{t_0}\big)\d\mu_{t_0}\Big| \\
& \leq \int_{t_0}^{t_1}\int\big(|\d f_{t_1}||X_t| + c|\Delta f_{t_1}|\big)\d\mu_t\dt + \iint_{t_0}^{t_1}\Big|\ddt f_t\Big|\,\d t\,\d\mu_{t_0}
\end{split}
\]
and our assumptions on $(f_t)$. Now fix a point $t$ of differentiability for $(f_t)$ and observe that the fact that $\frac{f_{t+h}-f_t}{h}$ strongly converges in $L^2(\Y)$ to $\ddt f_t$ and $\mu_{t+h}$ weakly converges to $\mu_t$ as $h \to 0$ and the densities are equibounded is sufficient to get
\[
\lim_{h\to 0}\int\frac{f_{t+h} - f_t}{h}\,\d\mu_{t+h} = \int \ddt f_t\,\d\mu_t = \lim_{h \to 0} \int \frac{f_{t+h}-f_t}{h}\,\d\mu_t.
\]
Hence the conclusion comes dividing by $h$ the trivial identity
\[
\begin{split}
\int f_{t+h}\,\d\mu_{t+h} - \int f_t\,\d\mu_t = & \int f_t\,\d\mu_{t+h} - \int f_t\,\d\mu_t + \int (f_{t+h} - f_t)\d\mu_t + \\
& \qquad + \int (f_{t+h}-f_t)\d\mu_{t+h} - \int (f_{t+h}-f_t)\d\mu_t
\end{split}
\]
and letting $h\to 0$.

The last statement is straightforward.
\end{proof}

We conclude with a threefold dynamical (but not yet variational) representation of the entropic cost.

\begin{Proposition}\label{pro:dyn}
With the same assumptions and notations as in Section \ref{sec:2.2}, for any $\eps>0$ the following holds:
\begin{equation}\label{eq:dyn}
\begin{split}
\eps\mathscr{I}_\eps(\mu_0,\mu_1) 
& = \frac{\eps}{2}\Big(H(\mu_0\,|\,\mm) + H(\mu_1\,|\,\mm)\Big) + \iint_0^1\Big( \frac{|\nabla\vartheta_t^\eps|^2}{2} + \frac{\eps^2}{8}|\nabla\log\rho_t^\eps|^2 \Big)\rho_t^\eps\dt\d\mm \\
& = \eps H(\mu_0\,|\,\mm) + \iint_0^1\frac{|\nabla\psi_t^\eps|^2}{2}\rho_t^\eps\dt\d\mm \\
& = \eps H(\mu_1\,|\,\mm) + \iint_0^1\frac{|\nabla\varphi_t^\eps|^2}{2}\rho_t^\eps\dt\d\mm.
\end{split}
\end{equation}
\end{Proposition}

\begin{proof}
Fix $\eps > 0$ and let us prove the first identity in \eqref{eq:dyn}. To this aim fix $\bar{x} \in \X$, $R>0$ and let $\nchi_R$ be a cut-off function as in Lemma \ref{lem:cutoff}; recalling that $(\rho_t^\eps) \in AC([0,1],L^2(\X))$ and for all compact set $\mathcal{C} \subset (0,1)$ there exists $M>0$ such that $(\vartheta_t^\eps) \in AC(\mathcal{C},W^{1,2}(\X,e^{-V}\mm))$ with $V = M\sfd^2(\cdot,\bar{x})$, we see that $t \mapsto \int\nchi_R\vartheta_t^\eps\rho_t^\eps\,\d\mm$ belongs to $AC_{loc}((0,1))$ with
\[
\begin{split}
\ddt\int\nchi_R\vartheta_t^\eps\rho_t^\eps\,\d\mm = & \int\nchi_R\Big(-\frac{|\nabla\vartheta_t^\eps|^2}{2} - \frac{\eps^2}{4}\Delta\log\rho_t^\eps - \frac{\eps^2}{8}|\nabla\log\rho_t^\eps|^2\Big)\rho_t^\eps\,\d\mm \\ & - \int\nchi_R\vartheta_t^\eps{\rm div}(\rho_t^\eps\nabla\vartheta_t^\eps)\d\mm \qquad {\rm a.e.}\ t\in(0,1).
\end{split}
\]
Integration by parts formula and integration in time on $[\delta,1-\delta]$ with $\delta \in (0,1/2)$ then yield
\[
\begin{split}
\int\nchi_R\vartheta_{1-\delta}^\eps\rho_{1-\delta}^\eps&\,\d\mm - \int\nchi_R\vartheta_{\delta}^\eps\rho_{\delta}^\eps\,\d\mm = \iint_{\delta}^{1-\delta}\nchi_R\Big(\frac{|\nabla\vartheta_t^\eps|^2}{2} + \frac{\eps^2}{8}|\nabla\log\rho_t^\eps|^2\Big)\rho_t^\eps\,\dt\d\mm \\ & + \iint_{\delta}^{1-\delta}\langle\nabla\nchi_R,\nabla\rho_t^\eps\rangle\,\dt\d\mm + \iint_{\delta}^{1-\delta}\vartheta_t^\eps\langle\nabla\nchi_R,\nabla\vartheta_t^\eps\rangle \rho_t^\eps\,\dt\d\mm.
\end{split}
\]
We claim that the limit as $R \to \infty$ can be carried under the integral signs. For the first summand on the right-hand side this is true by monotonicity, for all the other terms this follows from the dominated convergence theorem. Indeed $\vartheta_{\delta}^\eps\rho_{\delta}^\eps, \vartheta_{1-\delta}^\eps\rho_{1-\delta}^\eps \in L^1(\X)$ by Lemma \ref{lem:continuity}; $|\nabla\rho_t^\eps|$ and $\vartheta_t^\eps|\nabla\vartheta_t^\eps|\rho_t^\eps$ are, locally in t, uniformly bounded by an $L^1(\X)$ function, since
\[
|\nabla\rho_t^\eps| = \rho_t^\eps |\nabla\log\rho_t^\eps|, \qquad |\vartheta_t^\eps||\vartheta_t^\eps\rho_t^\eps| \leq \frac{1}{2}\rho_t^\eps \big(|\vartheta_t^\eps|^2 + |\nabla\vartheta_t^\eps|^2\big)
\]
and because of Lemma \ref{lem:continuity} again; $|\nchi_R|,|\nabla\nchi_R|$ are uniformly bounded in $L^\infty(\X)$ w.r.t.\ $R$ and converge $\mm$-a.e.\ to $1,0$ respectively by construction. Thus, we obtain
\[
\int\vartheta_{1-\delta}^\eps\rho_{1-\delta}^\eps\,\d\mm - \int\vartheta_{\delta}^\eps\rho_{\delta}^\eps\,\d\mm = \iint_{\delta}^{1-\delta}\Big(\frac{|\nabla\vartheta_t^\eps|^2}{2} + \frac{\eps^2}{8}|\nabla\log\rho_t^\eps|^2\Big)\rho_t^\eps\,\dt\d\mm.
\]
Now let $\delta \downarrow 0$: convergence of the right-hand side is trivial by monotonicity. For the left-hand side consider $t \mapsto \int\vartheta_t^\eps\rho_t^\eps\,\d\mm$, use the identity $\vartheta_t^\eps = \psi_t^\eps - \frac{\eps}{2}\log\rho_t^\eps$ and observe that
\[
t \mapsto \int\psi_t^\eps\rho_t^\eps\,\d\mm \quad\textrm{ and }\quad t \mapsto -\frac{\eps}{2}H(\mu_t^\eps \,|\, \mm)
\]
are both continuous at $t=0$, the former by Lemma \ref{lem:continuity} and the latter by \eqref{eq:ent-continuity}. This implies that
\[
\lim_{\delta \downarrow 0}\int\vartheta_{\delta}^\eps\rho_{\delta}^\eps\,\d\mm = \int\psi_0^\eps\rho_0\,\d\mm - \frac{\eps}{2}H(\mu_0\,|\,\mm).
\]
The same argument with the identity $\vartheta_t^\eps = -\varphi_t^\eps + \frac{\eps}{2}\log\rho_t^\eps$ allows us to handle $\int\vartheta_{1-\delta}^\eps\rho_{1-\delta}^\eps\,\d\mm$ too, so that
\[
\begin{split}
- \int\psi_0^\eps\rho_0\,\d\mm & -\int\varphi_1^\eps\rho_1\,\d\mm + \frac{\eps}{2}\Big(H(\mu_0\,|\,\mm) + H(\mu_1\,|\,\mm)\Big) \\ & = \iint_0^1\Big(\frac{|\nabla\vartheta_t^\eps|^2}{2} + \frac{\eps^2}{8}|\nabla\log\rho_t^\eps|^2\Big)\rho_t^\eps\,\dt\d\mm.
\end{split}
\]
Thanks to the identity $\varphi_0^\eps + \psi_0^\eps = \eps\log\rho_0$ in $\supp(\mu_0)$ and the integrability of $\psi_0^\eps\rho_0, \rho_0\log\rho_0$ we deduce that $\varphi_0^\eps\rho_0 \in L^1(\X)$ too. An analogous statement holds in $t=1$ and thus the previous identity is in turn equivalent to
\[
\begin{split}
\int\varphi_0^\eps\rho_0\,\d\mm & + \int\psi_1^\eps\rho_1\,\d\mm - \frac{\eps}{2}\Big(H(\mu_0\,|\,\mm) + H(\mu_1\,|\,\mm)\Big) \\ & = \iint_0^1\Big(\frac{|\nabla\vartheta_t^\eps|^2}{2} + \frac{\eps^2}{8}|\nabla\log\rho_t^\eps|^2\Big)\rho_t^\eps\,\dt\d\mm
\end{split}
\]
and now it is sufficient to observe that by \eqref{eq:def01}
\begin{equation}\label{eq:entcost}
\eps\mathscr{I}_\eps(\mu_0,\mu_1) = \eps H\big( f^\eps \otimes g^\eps\hR^{\eps/2} \,|\, \hR^{\eps/2}\big) = \int\varphi_0^\eps\d\mu_0 + \int\psi_1^\eps\d\mu_1.
\end{equation}
For the second and third identities in \eqref{eq:dyn}, the argument closely follows the one we have just presented. Indeed, it is just a matter of computation to rewrite the continuity equation solved by $(\rho_t^\eps,\vartheta_t^\eps)$ as forward and backward Fokker-Planck equations with velocity fields given by $\nabla\psi_t^\eps$ and $\nabla\varphi_t^\eps$ respectively, i.e.
\begin{subequations}
\begin{align}
\label{eq:ffpe}
\ddt\rho_t^\eps + {\rm div}(\rho_t^\eps\nabla\psi_t^\eps) & = \frac{\eps}{2}\Delta\rho_t^\eps \\
\label{eq:bfpe}
-\ddt\rho_t^\eps + {\rm div}(\rho_t^\eps\nabla\varphi_t^\eps) & = \frac{\eps}{2}\Delta\rho_t^\eps
\end{align}
\end{subequations}
where the time derivatives are meant in the $L^2(\X)$ sense as in \eqref{eq:pde1}. Therefore, arguing as above it is not difficult to see that
\[
\int\varphi_1^\eps\d\mu_1 - \int\varphi_0^\eps\d\mu_0 = -\iint_0^1\frac{|\nabla\varphi_t^\eps|^2}{2}\rho_t^\eps\dt\d\mm
\]
and an analogous identity holds true for $\psi_t^\eps$. Finally using the identities $\varphi_0^\eps + \psi_0^\eps = \eps\log\rho_0$ in $\supp(\mu_0)$ and $\varphi_1^\eps + \psi_1^\eps = \eps\log\rho_1$ in $\supp(\mu_1)$, the entropic cost can be rewritten as
\[
\eps\mathscr{I}_\eps(\mu_0,\mu_1) = \eps H(\mu_0\,|\,\mm) + \int\psi_1^\eps\d\mu_1 - \int\psi_0^\eps\d\mu_0 = \eps H(\mu_1\,|\,\mm) + \int\varphi_0^\eps\d\mu_0 - \int\varphi_1^\eps\d\mu_1,
\]
whence the conclusion.
\end{proof}

\section{Dynamical and dual representations of the entropic cost}

From Proposition \ref{pro:dyn} we notice that the entropic cost can be expressed as the evaluation of an action functional in three different ways: as a purely kinetic energy evaluated along $(\rho_t^\eps,\psi_t^\eps)$ and $(\rho_t^\eps,\varphi_t^\eps)$ or as the sum of kinetic energy and Fisher information along $(\rho_t^\eps,\vartheta_t^\eps)$. As we have just seen, three different `PDEs' are associated to these couples, so that three different minimization problems can be introduced, as already discussed in the Introduction.

The first of them reads as
\[
\inf_{(\eta,v)}\iint_0^1\Big( \frac{|v_t|^2}{2} + \frac{\eps^2}{8}|\nabla\log\eta_t|^2 \Big)\eta_t\dt\d\mm,
\]
the infimum being taken among all solutions $(\eta\mm,v)$ of the continuity equation in the sense of Definition \ref{def:fpe}. In line with the smooth theory, we are now able to prove that, if we add $\frac{\eps}{2}(H(\mu_0\,|\,\mm) + H(\mu_1\,|\,\mm))$, this infimum coincides with $\eps\mathscr{I}_\eps(\mu_0,\mu_1)$, thus providing a first variational representation of the entropic cost and extending \eqref{eq:intbbs2} to the $\RCD$ framework. Moreover, the infimum is attained if and only if $(\eta_t,v_t) = (\rho_t^\eps,\nabla\vartheta_t^\eps)$. We remark that the uniqueness of the minimizers is not stated in \cite{MikamiThieullen08}, \cite{ChGePa16} and \cite{GLR15}.

\begin{Theorem}[Benamou-Brenier formula for the entropic cost, 1st form]\label{thm:main1}
With the same assumptions and notations as in Section \ref{sec:2.2}, for any $\eps>0$ the following holds:
\begin{equation}\label{eq:bbs}
\eps\mathscr{I}_\eps(\mu_0,\mu_1) = \frac{\eps}{2}\Big(H(\mu_0\,|\,\mm) + H(\mu_1\,|\,\mm)\Big) + \min_{(\eta,v)}\bigg\{\iint_0^1\Big( \frac{|v_t|^2}{2} + \frac{\eps^2}{8}|\nabla\log\eta_t|^2 \Big)\eta_t\dt\d\mm\bigg\}
\end{equation}
where the minimum is taken among all couples $(\eta\mm,v)$ solving the continuity equation in the sense of Definition \ref{def:fpe} under the constraint $\eta_0\mm = \mu_0$ and $\eta_1\mm = \mu_1$; moreover, the minimum is attained if and only if $(\eta_t,v_t) = (\rho_t^\eps,\nabla\vartheta_t^\eps)$.
\end{Theorem}

\begin{proof}

\noindent{\bf Inequality $\geq$ in \eqref{eq:bbs}}. Proposition \ref{pro:dyn}, the third in \eqref{eq:pde1} and the last in \eqref{eq:finitekin} imply that $(\rho_t^\eps,\nabla\vartheta_t^\eps)$ solves the continuity equation in the sense of Definition \ref{def:fpe}. 

\noindent{\bf Inequality $\leq$ in \eqref{eq:bbs}}. Start noticing that the assumptions on $\mu_0,\mu_1$ grant that $\mathscr{I}_\eps(\mu_0,\mu_1)$ is finite. Thus, given a solution $(\eta,v)$ of the continuity equation in the sense of the statement, without loss of generality we can assume that
\begin{equation}\label{eq:finito}
\iint_0^1\Big( \frac{|v_t|^2}{2} + \frac{\eps^2}{8}|\nabla\log\eta_t|^2 \Big)\eta_t\dt\d\mm < +\infty.
\end{equation}
Now fix $x \in \X$, $R>0$ and pick a cut-off function $\nchi_R$ as in Lemma \ref{lem:cutoff}; take also $\delta>0$ and define for all $t \in [0,1]$
\[
\varphi_t^{\eps,\delta} := \eps\log(f_t^\eps + \delta) \qquad \psi_t^{\eps,\delta} := \eps\log(g_t^\eps + \delta) \qquad \vartheta_t^{\eps,\delta} := \frac{1}{2}\big(\psi_t^{\eps,\delta} - \varphi_t^{\eps,\delta}\big).
\]
By Lemma \ref{lem:hjb} $(\nchi_R\vartheta_t^{\eps,\delta}) \in AC_{loc}((0,1),L^2(\X)) \cap L^\infty_{loc}((0,1),W^{1,2}(\X))$. Thus, given $t_0,t_1 \in (0,1)$ with $t_0 < t_1$, Lemma \ref{lem:dermiste} applies to $(\eta_t\mm)$ and $t \mapsto \nchi_R\vartheta_t^{\eps,\delta}$ on $[t_0,t_1]$, whence
\[
\frac{\d}{\d s}\Big(\int\nchi_R\vartheta_s^{\eps,\delta} \eta_s\d\mm\Big)\restr{s=t} = \int \nchi_R\big(\frac{\d}{\d s}\vartheta_s^{\eps,\delta}\restr{s=t}\big)\eta_t\d\mm + \frac{\d}{\d s}\Big(\int\nchi_R\vartheta_t^{\eps,\delta} \eta_s\d\mm \Big)\restr{s=t}
\]
for a.e.\ $t \in [t_0,t_1]$. For the first term on the right-hand side, the fact that $\vartheta_t^{\eps,\delta} = (\psi_t^{\eps,\delta} - \varphi_t^{\eps,\delta})/2$ and the `PDEs' solved by $\psi_t^{\eps,\delta},\varphi_t^{\eps,\delta}$ (namely \eqref{eq:hjb-eta} up to rescaling and change of sign) yield
\[
\begin{split}
\int \nchi_R\big(\frac{\d}{\d s}\vartheta_s^{\eps,\delta}\restr{s=t}\big)\eta_t\d\mm & = -\int\nchi_R\Big(\frac{|\nabla\psi_t^{\eps,\delta}|^2}{4} + \frac{\eps}{4}\Delta\psi_t^{\eps,\delta} + \frac{|\nabla\varphi_t^{\eps,\delta}|^2}{4} + \frac{\eps}{4}\Delta\varphi_t^{\eps,\delta}\Big)\eta_t\d\mm \\ & = \int\nchi_R\Big(-\frac{|\nabla\psi_t^{\eps,\delta}|^2}{4} - \frac{|\nabla\varphi_t^{\eps,\delta}|^2}{4} + \frac{\eps}{4}\langle\nabla(\psi_t^{\eps,\delta} + \varphi_t^{\eps,\delta}),\nabla\log\eta_t\rangle\Big)\eta_t\d\mm \\ & \, \quad + \frac{\eps}{4}\int \langle\nabla(\psi_t^{\eps,\delta} + \varphi_t^{\eps,\delta}),\nabla\nchi_R\rangle\,\eta_t\d\mm
\end{split}
\]
and by Young's inequality
\begin{equation}\label{eq:young1}
\eps\langle\nabla(\psi_t^{\eps,\delta} + \varphi_t^{\eps,\delta}),\nabla\log\eta_t\rangle \leq \frac{1}{2}|\nabla(\psi_t^{\eps,\delta} + \varphi_t^{\eps,\delta})|^2 + \frac{\eps^2}{2}|\nabla\log\eta_t|^2.
\end{equation}
On the other hand, the fact that $(\eta,v)$ is a solution of the continuity equation and $\nchi_R\vartheta_t^{\eps,\delta} \in W^{1,2}(\X)$ imply by Lemma \ref{lem:dermiste} that
\[
\frac{\d}{\d s}\Big(\int\nchi_R\vartheta_t^{\eps,\delta} \eta_s\d\mm \Big)\restr{s=t} = \frac{1}{2}\int\Big(\nchi_R\langle\nabla(\psi_t^{\eps,\delta} - \varphi_t^{\eps,\delta}),v_t\rangle + (\psi_t^{\eps,\delta} - \varphi_t^{\eps,\delta}) \langle\nabla\nchi_R,v_t\rangle\Big)\eta_t\d\mm
\]
and by Young's inequality
\begin{equation}\label{eq:young2}
\langle\nabla(\psi_t^{\eps,\delta} - \varphi_t^{\eps,\delta}),v_t\rangle \leq \frac{1}{4}|\nabla( \psi_t^{\eps,\delta} - \varphi_t^{\eps,\delta})|^2 + |v_t|^2.
\end{equation}
Plugging these observations together and integrating over $[t_0,t_1]$ we deduce that
\[
\begin{split}
\int\nchi_R\vartheta_{t_1}^{\eps,\delta}\eta_{t_1}\,\d\mm - \int\nchi_R\vartheta_{t_0}^{\eps,\delta}\eta_{t_0}\,\d\mm & \leq \iint_{t_0}^{t_1}\nchi_R\Big(\frac{|v_t|^2}{2} + \frac{\eps^2}{8}|\nabla\log\eta_t|^2\Big)\eta_t\dt\d\mm \\ & \, \quad + \frac{\eps}{4}\iint_{t_0}^{t_1} \langle\nabla(\psi_t^{\eps,\delta} + \varphi_t^{\eps,\delta}),\nabla\nchi_R\rangle\,\eta_t\dt\d\mm \\ & \, \quad + \frac{1}{2}\iint_{t_0}^{t_1} (\psi_t^{\eps,\delta} - \varphi_t^{\eps,\delta}) \langle\nabla\nchi_R,v_t\rangle\,\eta_t\dt\d\mm.
\end{split}
\]
Now notice that
\[
\lim_{R \to \infty} \iint_{t_0}^{t_1}\nchi_R\Big(\frac{|v_t|^2}{2} + \frac{\eps^2}{8}|\nabla\log\eta_t|^2\Big)\eta_t\dt\d\mm = \iint_{t_0}^{t_1} \Big(\frac{|v_t|^2}{2} + \frac{\eps^2}{8}|\nabla\log\eta_t|^2\Big)\eta_t\dt\d\mm
\]
by monotonicity. Moreover, on the left-hand side $\vartheta_t^{\eps,\delta}\eta_t \in L^1(\X)$ for $t = t_0,t_1$ by (iv) in Lemma \ref{lem:hjb} and the very definition of $\vartheta_t^{\eps,\delta}$; on the right-hand side we know that $\| |\nabla\nchi_R| \|_{L^\infty(\X)}$ is uniformly bounded w.r.t.\ $R$, $|\nabla(\psi_t^{\eps,\delta} + \varphi_t^{\eps,\delta})|\eta_t \in L^1([t_0,t_1] \times \X,\dt \otimes \mm)$ by Lemma \ref{lem:hjb} again,
\[
|(\psi_t^{\eps,\delta} - \varphi_t^{\eps,\delta}) \langle\nabla\nchi_R,v_t\rangle\,\eta_t| \leq \frac{1}{2}|\nabla\nchi_R|\eta_t\big( |\psi_t^{\eps,\delta} - \varphi_t^{\eps,\delta}|^2 + |v_t|^2\big)
\]
and $|v_t|^2\eta_t \in L^1([0,1] \times \X,\dt \otimes \mm)$ by \eqref{eq:finito}. Thus the dominated convergence theorem applies to all the remaining terms and, since $|\nabla\nchi_R| \to 0$ $\mm$-a.e.\ as $R \to \infty$, this implies
\begin{equation}\label{eq:calibro35}
\int\vartheta_{t_1}^{\eps,\delta}\eta_{t_1}\,\d\mm - \int\vartheta_{t_0}^{\eps,\delta}\eta_{t_0}\,\d\mm \leq \iint_{t_0}^{t_1}\Big(\frac{|v_t|^2}{2} + \frac{\eps^2}{8}|\nabla\log\eta_t|^2\Big)\eta_t\dt\d\mm.
\end{equation}
Passing to the limit as $t_0 \downarrow 0$, the convergence of the right-hand side is trivial by monotonicity. As regards the left-hand side, notice that
\begin{equation}\label{eq:t0}
\bigg|\int\psi_{t_0}^{\eps,\delta}\eta_{t_0}\,\d\mm - \int\psi_0^{\eps,\delta}\,\d\mu_0 \bigg| \leq \int |\psi_{t_0}^{\eps,\delta} - \psi_0^{\eps,\delta}|\eta_{t_0}\,\d\mm + \bigg|\int\psi_0^{\eps,\delta}\eta_{t_0}\,\d\mm - \int\psi_0^{\eps,\delta}\,\d\mu_0 \bigg|
\end{equation}
and the fact that $\log$ is Lipschitz on $[\delta,\infty)$ together with the fact that $\eta_t \leq C$ for all $t \in [0,1]$ entails
\[
\int |\psi_{t_0}^{\eps,\delta} - \psi_0^{\eps,\delta}|\eta_{t_0}\,\d\mm \leq C'\|g_{t_0}^\eps - g_0^\eps\|_{L^2(\X)},
\]
so that the first term on the right-hand side of \eqref{eq:t0} vanishes as $t_0 \downarrow 0$, since $t \mapsto g_t^\eps$ is $L^2$-continuous. As regards the second one, it also disappears: indeed $g_0^\eps \in C_b(\X)$ (as a consequence of $g^\eps \in L^\infty(\X)$ with compact support, \eqref{eq:lipreg} and the maximum principle for the heat equation) so that $\psi_0^{\eps,\delta} \in C_b(\X)$ too, and $\eta_{t_0}\mm \rightharpoonup \mu_0$ as $t_0 \downarrow 0$ by definition.

For $\varphi_t^{\eps,\delta}$ we argue in an analogous way: we write
\[
\bigg|\int\varphi_{t_0}^{\eps,\delta}\eta_{t_0}\,\d\mm - \int\varphi_0^{\eps,\delta}\,\d\mu_0 \bigg| \leq \int |\varphi_{t_0}^{\eps,\delta} - \varphi_0^{\eps,\delta}|\eta_{t_0}\,\d\mm + \bigg|\int\varphi_0^{\eps,\delta}\eta_{t_0}\,\d\mm - \int\varphi_0^{\eps,\delta}\,\d\mu_0 \bigg|
\]
and get rid of the first term on the right-hand side as just done for $\psi_t^{\eps,\delta}$. For the second one, let us stress that in general neither $f^\eps$ nor $\varphi_0^{\eps,\delta}$ belong to $C_b(\X)$ and thus narrow convergence can not be applied. However, $f^\eps$ has compact support, since so does $\rho_0$: this means that $\varphi_0^{\eps,\delta}$ is constant outside a bounded set and thus for all $\alpha > 0$ we can find $h \in C_b(\X)$ such that $\|\varphi_0^{\eps,\delta} - h\|_{L^1(\X)} < \alpha$; therefore, using again $\eta_t \leq C$, we get
\[
\begin{split}
\bigg|\int\varphi_0^{\eps,\delta}\eta_{t_0}\,\d\mm - \int\varphi_0^{\eps,\delta}\,\d\mu_0 \bigg| & \leq \int |\varphi_0^{\eps,\delta} - h|\eta_{t_0}\,\d\mm + \bigg|\int h\eta_{t_0}\,\d\mm - \int h\,\d\mu_0 \bigg| + \int |\varphi_0^{\eps,\delta} - h|\,\d\mu_0 \\ & \leq 2C\alpha + \bigg|\int h\eta_{t_0}\,\d\mm - \int h\,\d\mu_0 \bigg|
\end{split}
\]
and the arbitrariness of $\alpha$ together with $\eta_{t_0}\mm \rightharpoonup \mu_0$ as $t_0 \downarrow 0$ allows us to conclude in the same way as for $\psi_t^{\eps,\delta}$. Since $\vartheta_t^{\eps,\delta} := (\psi_t^{\eps,\delta} - \varphi_t^{\eps,\delta})/2$, this implies
\[
\lim_{t_0 \downarrow 0}\int\vartheta_{t_0}^{\eps,\delta}\eta_{t_0}\,\d\mm = \int\vartheta_0^{\eps,\delta}\,\d\mu_0.
\]
In a completely analogous way we can handle the case $t_1 \uparrow 1$, whence from \eqref{eq:calibro35}
\begin{equation}\label{eq:quasibbs}
\int\vartheta_1^{\eps,\delta}\,\d\mu_1 - \int\vartheta_0^{\eps,\delta}\,\d\mu_0 \leq \iint_0^1\Big(\frac{|v_t|^2}{2} + \frac{\eps^2}{8}|\nabla\log\eta_t|^2\Big)\eta_t\dt\d\mm.
\end{equation}
Now use the identity $\vartheta_1^{\eps,\delta} = \psi_1^{\eps,\delta} - \frac{\eps}{2}\log((f_1^\eps + \delta)(g^\eps + \delta))$ and observe that by monotonicity
\[
\lim_{\delta \downarrow 0}\int\psi_1^{\eps,\delta}\,\d\mu_1 = \int\psi_1^\eps\,\d\mu_1, \qquad\qquad \lim_{\delta \downarrow 0}\int\log((f_1^\eps + \delta)(g^\eps + \delta))\d\mu_1 = H(\mu_1\,|\,\mm).
\]
Moreover both limits are finite, since the assumptions on $\rho_1$ grant that $H(\mu_1\,|\,\mm) < \infty$, while the first one can be rewritten as $\eps H(\mu_1\,|\,\mm) - \int\varphi_1^\eps\,\d\mu_1$ and $\varphi_1^\eps\rho_1 \in L^1(\X)$ by Lemma \ref{lem:continuity}. From the identity $\vartheta_0^{\eps,\delta} = -\varphi_0^{\eps,\delta} + \frac{\eps}{2}\log((f^\eps + \delta)(g_0^\eps + \delta))$ an analogous statement holds for $\int\vartheta_0^{\eps,\delta}\,\d\mu_0$ as $\delta \downarrow 0$ too, so that combining these remarks with \eqref{eq:quasibbs} and \eqref{eq:entcost} we get
\[
\eps\mathscr{I}_\eps(\mu_0,\mu_1) \leq \frac{\eps}{2}\Big(H(\mu_0\,|\,\mm) + H(\mu_1\,|\,\mm)\Big) + \iint_0^1\Big( \frac{|v_t|^2}{2} + \frac{\eps^2}{8}|\nabla\log\eta_t|^2 \Big)\eta_t\dt\d\mm,
\]
whence the conclusion by taking the infimum.

\noindent{\bf Uniqueness of the minimizer}. As a first step, notice that if $(\eta_t\mm,v_t)$ solves the continuity equation as in Definition \ref{def:fpe} and we introduce the `momentum' variable $m_t := \eta_t v_t$, then
\[
\Gamma := \{ (\eta,m) \,:\, (\eta\mm,v) \textrm{ solves the continuity equation},\, \eta_0\mm = \mu_0,\, \eta_1\mm = \mu_1 \}
\]
is closed w.r.t.\ convex combination. Secondly, the function $\Phi : [0,\infty) \times \R \to [0,\infty]$ defined by
\[
\Phi(x,y) := \left\{\begin{array}{ll}
\displaystyle{\frac{y^2}{x}} & \qquad \text{ if } x > 0, \\
0 & \qquad \text{ if } x=y=0,\\
+\infty & \qquad \text{ otherwise } 
\end{array}\right.
\]
is convex and lower semicontinuous. Therefore the functionals $\mathscr{K},\mathscr{F},\mathscr{A} : \Gamma \to [0,\infty]$ defined as
\[
\mathscr{K}(\eta,m) := \frac{1}{2}\iint_0^1 \frac{|m_t|^2}{\eta_t}\,\dt\d\mm, \quad \mathscr{F}(\eta,m) := \frac{\eps^2}{8} \iint_0^1 \frac{|\nabla\eta_t|^2}{\eta_t}\,\dt\d\mm, \quad \mathscr{A} := \mathscr{K} + \mathscr{F}
\]
(in case the set of $t$'s where $\eta_t \notin W^{1,2}_{loc}(\X)$ has positive $\mathcal{L}^1$-measure, we set $\mathscr{F}(\eta,m) := +\infty$) are convex too. Thus, if $(\overline{\eta},\overline{v})$ is a minimizer for \eqref{eq:bbs} and $\overline{m}_t := \overline{\eta}_t\overline{v}_t$, then by Proposition \ref{pro:dyn} it follows that $\mathscr{A}(\overline{\eta},\overline{m}) = \mathscr{A}(\rho^\eps,\rho^\eps\nabla\vartheta^\eps)$ and so, by convexity of $\mathscr{A}$,
\[
\mathscr{A}(\eta^\lambda,m^\lambda) = (1-\lambda)\mathscr{A}(\overline{\eta},\overline{m}) + \lambda \mathscr{A}(\rho^\eps,\rho^\eps\nabla\vartheta^\eps) \qquad \forall \lambda \in (0,1),
\]
where $\eta_t^\lambda := (1-\lambda)\eta_t + \lambda\rho_t^\eps$ and $m_t^\lambda := (1-\lambda)m_t + \lambda\rho_t^\eps\nabla\vartheta_t^\eps$. As a byproduct the same identity holds with $\mathscr{F}$ instead of $\mathscr{A}$, since both $\mathscr{K}$ and $\mathscr{F}$ are convex and $\mathscr{A} = \mathscr{K} + \mathscr{F}$, whence
\begin{equation}\label{eq:linearity}
\Phi(\eta_t^\lambda,|\nabla\eta_t^\lambda|) = (1-\lambda)\Phi(\overline{\eta}_t,|\nabla\overline{\eta}_t|) + \lambda \Phi(\rho_t^\eps,|\nabla\rho_t^\eps|).
\end{equation}
Taking into account the fact that $\Phi$ is linear only along the lines passing through the origin, we deduce that $(\overline{\eta}_t,|\nabla\overline{\eta}_t|)$ and $(\rho_t^\eps,|\nabla\rho_t^\eps|)$ are collinear. Since $\rho_t^\eps > 0$, it is not restrictive to assume that the collinearity condition can be expressed as follows: for all $t \in (0,1)$ and $\mm$-a.e.\ $x \in \X$ there exists $\alpha_{x,t} \geq 0$ such that
\begin{equation}\label{eq:collinear}
(\overline{\eta}_t,|\nabla\overline{\eta}_t|) = \alpha_{x,t}(\rho_t^\eps,|\nabla\rho_t^\eps|).
\end{equation}
Since $\alpha_{x,t} = \overline{\eta}_t/\rho_t^\eps$ with $\rho_t^\eps \in W^{1,2}(\X)$ locally bounded away from 0 and $\overline{\eta}_t \in W^{1,2}_{loc}(\X)$ for a.e.\ $t$, this implies that $\alpha_{x,t} \in W^{1,2}_{loc}(\X)$ for a.e.\ $t$. Furthermore, on the one hand by the very definition of $\eta_t^\lambda$ together with \eqref{eq:collinear} we have
\[
\eta_t^\lambda = \big( (1-\lambda)\alpha_{x,t} + \lambda\big)\rho_t^\eps
\]
and on the other hand \eqref{eq:linearity} also implies that $(\eta_t^\lambda,|\nabla\eta_t^\lambda|)$ is collinear with $(\rho_t^\eps,|\nabla\rho_t^\eps|)$: this implies
\[
(\eta_t^\lambda,|\nabla\eta_t^\lambda|) = \big( (1-\lambda)\alpha_{x,t} + \lambda\big)(\rho_t^\eps,|\nabla\rho_t^\eps|).
\]
Considering the gradient of the identity on the first coordinate we see that
\[
|\nabla\eta_t^\lambda|^2 = \big( (1-\lambda)\alpha_{x,t} + \lambda\big)^2|\nabla\rho_t^\eps|^2 + (1-\lambda)^2(\rho_t^\eps)^2|\nabla\alpha_{x,t}|^2 + 2(1-\lambda)\big( (1-\lambda)\alpha_{x,t} + \lambda\big) \rho_t^\eps \langle\nabla\rho_t^\eps,\nabla\alpha_{x,t}\rangle
\]
and plugging the identity on the second coordinate therein this becomes
\begin{equation}\label{eq:alpha}
|\nabla\alpha_{x,t}|^2 = -2\Big(\alpha_{x,t} + \frac{\lambda}{1-\lambda}\Big)\langle\nabla\log\rho_t^\eps,\nabla\alpha_{x,t}\rangle, \qquad \forall \lambda \in (0,1).
\end{equation}
Since $\alpha_{x,t}$ does not depend on $\lambda$, this can be true only if $\langle\nabla\log\rho_t^\eps,\nabla\alpha_{x,t}\rangle = 0$ $\mm$-a.e.\, thus forcing also the left-hand side of \eqref{eq:alpha} to vanish, namely $\alpha_{x,t}$ to be constant. Since $\overline{\eta}_t,\rho_t^\eps$ are probability densities, we conclude that $\alpha_{x,t} \equiv 1$.

Thus $\overline{\eta}_t = \rho_t^\eps$ for all $t \in [0,1]$ and now it is sufficient to use the strict convexity of $(v_t) \mapsto \iint_0^1 |v_t|^2\rho_t^\eps\,\dt\d\mm$ to conclude that $\overline{v}_t = \nabla\vartheta_t^\eps$ for all $t \in [0,1]$.
\end{proof}

\begin{Remark}[Uniqueness of the minimizer]\label{rem:1}
{\rm
As proved in \cite{Tamanini17}, if $\X$ is assumed to be a compact $\RCD^*(K,N)$ space and $\rho_0,\rho_1$ are sufficiently regular (i.e.\ they belong to $\testipp\X$ with the notation used therein), then the fact that the minimum is attained in \eqref{eq:bbs} if and only if $(\eta_t,v_t) = (\rho_t^\eps,\nabla\vartheta_t^\eps)$ is straightforward. This is essentially due to the fact that no cut-off argument is needed and $(\vartheta_t^\eps) \in AC([0,1],W^{1,2}(\X))$, so that $t \mapsto \int\vartheta_t^\eps\eta_t\,\d\mm$ belongs to $AC([0,1])$, no limit as $t_0 \downarrow 0, t_1 \uparrow 1$ and $\delta \downarrow 0$ appears in the previous proof and by the case of equality in \eqref{eq:young1} and \eqref{eq:young2} the infimum is attained if and only if
\[
\nabla(\psi_t^\eps + \varphi_t^\eps) = \eps\nabla\log\eta_t \qquad\textrm{and}\qquad v_t = \frac{1}{2}\nabla(\psi_t^\eps - \varphi_t^\eps),
\]
which completely characterizes the optimal pair $(\eta_t,v_t)$.
}\fr
\end{Remark}

The second and third dynamical variational representation of the entropic cost (and thus the non-smooth analogue of \eqref{eq:intbbs}) are the content of the next result.

\begin{Theorem}[Benamou-Brenier formula for the entropic cost, 2nd form]\label{thm:main1bis}
With the same assumptions and notations as in Section \ref{sec:2.2}, for any $\eps>0$ the following holds:
\begin{itemize}
\item[(i)]
\begin{equation}\label{eq:bbs2}
\eps\mathscr{I}_\eps(\mu_0,\mu_1) = \eps H(\mu_0\,|\,\mm) + \min_{(\nu^+,v^+)}\bigg\{\int_0^1\int\frac{|v^+_t|^2}{2}\d\nu^+_t\dt\bigg\}
\end{equation}
where the minimum is taken among all couples $(\nu^+,v^+)$ solving the forward Fokker-Planck equation in the sense of Definition \ref{def:fpe} with $c=\eps/2$ under the constraint $\nu_0 = \mu_0$ and $\nu_1 = \mu_1$; the minimum is attained at $(\rho_t^\eps\mm,\nabla\psi_t^\eps)$.
\item[(ii)]
\begin{equation}\label{eq:bbs3}
\eps\mathscr{I}_\eps(\mu_0,\mu_1) = \eps H(\mu_1\,|\,\mm) + \min_{(\nu^-,v^-)}\bigg\{\int_0^1\int\frac{|v^-_t|^2}{2}\d\nu^-_t\dt\bigg\}
\end{equation}
where the minimum is taken among all couples $(\nu^-,v^-)$ solving the backward Fokker-Planck equation in the sense of Definition \ref{def:fpe} with $c=\eps/2$ under the constraint $\nu_0 = \mu_0$ and $\nu_1 = \mu_1$; the minimum is attained at $(\rho_t^\eps\mm,\nabla\varphi_t^\eps)$.
\end{itemize}
\end{Theorem}

\begin{proof}
It is sufficient to prove $(i)$, as $(ii)$ follows by swapping $\mu_0$ and $\mu_1$.

\noindent{\bf Inequality $\geq$ in \eqref{eq:bbs2}}. Proposition \ref{pro:dyn}, \eqref{eq:ffpe} and the second in \eqref{eq:finitekin} imply that $(\rho_t^\eps,\nabla\psi_t^\eps)$ solves the forward Fokker-Planck equation in the sense of Definition \ref{def:fpe}. 

\noindent{\bf Inequality $\leq$ in \eqref{eq:bbs2}}. For sake of notation, let us drop the apex $+$ in $(\nu^+,v^+)$. Fix $x \in \X$, $R>0$ and pick a cut-off function $\nchi_R$ as in Lemma \ref{lem:cutoff}; take also $\delta>0$ and define $\psi_t^{\eps,\delta}$ as in $(i)$. By Lemma \ref{lem:hjb} $(\nchi_R\psi_t^{\eps,\delta}) \in AC_{loc}([0,1),L^2(\X)) \cap L^\infty_{loc}([0,1),W^{1,2}(\X))$ with $(\Delta(\nchi_R\psi_t^{\eps,\delta})) \in L^\infty_{loc}([0,1),L^2(\X))$. Thus, given $t_1 \in [0,1)$, Lemma \ref{lem:dermiste} applies to $(\nu_t)$ and $t \mapsto \nchi_R\psi_t^{\eps,\delta}$ on $[0,t_1]$, whence
\[
\frac{\d}{\d s}\Big(\int\nchi_R\psi_s^{\eps,\delta}\,\d\nu_s\Big)\restr{s=t} = \int \nchi_R\big(\frac{\d}{\d s}\psi_s^{\eps,\delta}\restr{s=t}\big)\d\nu_t + \frac{\d}{\d s}\Big(\int\nchi_R\psi_t^{\eps,\delta} \d\nu_s \Big)\restr{s=t}
\]
for a.e.\ $t \in [0,t_1]$. For the first term on the right-hand side, the fact that $\psi_t^{\eps,\delta}$ solves \eqref{eq:hjb-eta} (up to rescaling and change of sign) yields
\[
\int \nchi_R\big(\frac{\d}{\d s}\psi_s^{\eps,\delta}\restr{s=t}\big)\d\nu_t = -\int\nchi_R\Big(\frac{|\nabla\psi_t^{\eps,\delta}|^2}{2} + \frac{\eps}{2}\Delta\psi_t^{\eps,\delta}\Big)\d\nu_t.
\]
On the other hand, the fact that $(\nu,v)$ is a solution of the forward Fokker-Planck equation with $c = \eps/2$ and $\nchi_R\psi_t^{\eps,\delta} \in D(\Delta)$ imply that
\[
\begin{split}
\frac{\d}{\d s}\Big(\int\nchi_R\psi_t^{\eps,\delta}\,\d\nu_s \Big)\restr{s=t} & = \int\Big(\nchi_R\langle\nabla\psi_t^{\eps,\delta},v_t\rangle + \psi_t^{\eps,\delta} \langle\nabla\nchi_R,v_t\rangle\Big)\d\nu_t \\ & \, \quad + \frac{\eps}{2}\int\Big(\nchi_R\Delta\psi_t^{\eps,\delta} + 2\langle\nabla\nchi_R,\nabla\psi_t^{\eps,\delta} \rangle + \psi_t^{\eps,\delta}\Delta\nchi_R\Big)\d\nu_t
\end{split}
\]
and by Young's inequality
\begin{equation}\label{eq:young3}
\langle\nabla\psi_t^{\eps,\delta},v_t\rangle \leq \frac{1}{2}|\nabla\psi_t^{\eps,\delta}|^2 + \frac{1}{2}|v_t|^2.
\end{equation}
Plugging these observations together and integrating over $[0,t_1]$ we deduce that
\[
\begin{split}
\int\nchi_R\psi_{t_1}^{\eps,\delta}\,\d\nu_{t_1} - \int\nchi_R\psi_0^{\eps,\delta}\,\d\mu_0 & \leq \int_0^{t_1}\int\nchi_R \frac{|v_t|^2}{2}\d\nu_t\dt + \int_0^{t_1}\int\psi_t^{\eps,\delta}\langle\nabla\nchi_R,v_t\rangle\d\nu_t\dt \\ & \, \quad + \frac{\eps}{2}\int_0^{t_1}\int \Big(\psi_t^{\eps,\delta}\Delta\nchi_R + 2\langle\nabla\psi_t^{\eps,\delta},\nabla\nchi_R\rangle \Big)\d\nu_t\dt.
\end{split}
\]
Since Lemma \ref{lem:cutoff} ensures that $\| \Delta\nchi_R \|_{L^\infty(\X)}$ is uniformly bounded w.r.t.\ $R$ and $t \mapsto \int |v_t|^2\d\nu_t$ belongs to $L^1(0,1)$ by Definition \ref{def:fpe}, the argument by dominated convergence explained in the proof of Theorem \ref{thm:main1} applies and thus, passing to the limit as $R \to \infty$, we get
\[
\int\psi_{t_1}^{\eps,\delta}\,\d\nu_{t_1} - \int\psi_0^{\eps,\delta}\,\d\mu_0 \leq \int_0^{t_1}\int\frac{|v_t|^2}{2}\d\nu_t\dt.
\]
Both limits as $t_1 \uparrow 1$ and $\delta \downarrow 0$ can also be handled as in Theorem \ref{thm:main1}, whence
\[
\int\psi_1^\eps\,\d\mu_1 - \int\psi_0^\eps\,\d\mu_0 \leq \int_0^1\int\frac{|v_t|^2}{2}\d\nu_t\dt,
\]
and now it is sufficient to use the identity $\psi_0^\eps = -\varphi_0^\eps + \eps\log\rho_0$ in $\supp(\mu_0)$ in conjunction with \eqref{eq:entcost} to get the conclusion.
\end{proof}

\begin{Remark}[Uniqueness of the drift]{\rm
In the case $\X$ is assumed to be a compact $\RCD^*(K,N)$ space and $\rho_0,\rho_1$ are sufficiently regular, then as proved in \cite{Tamanini17} the drift of the optimal couple is uniquely determined, namely: the minimum is attained in \eqref{eq:bbs2} (resp.\ \eqref{eq:bbs3}) if and only if $v_t = \nabla\psi_t^\eps$ (resp.\ $v_t = \nabla\varphi_t^\eps$). As for Remark \ref{rem:1}, this is essentially due to the fact that no cut-off argument is needed and $(\psi_t^\eps) \in AC([0,1],W^{1,2}(\X))$, so that $t \mapsto \int\psi_t^\eps\eta_t\,\d\mm$ belongs to $AC([0,1])$, no limit as $t_1 \uparrow 1$ and $\delta \downarrow 0$ appears in the previous proof and by the case of equality in \eqref{eq:young3} the infimum is attained if and only if $v_t = \nabla\psi_t^\eps$.
}\fr
\end{Remark}

As already suggested by the proof of Theorem \ref{thm:main1bis}, the duality between Hamilton-Jacobi and continuity equation that appears in optimal transport is here replaced by the duality between forward (resp.\ backward) Hamilton-Jacobi-Bellman and backward (resp.\ forward) Fokker-Planck equation. This will be the content of the next result.

\begin{Theorem}[HJB duality for the entropic cost]\label{thm:duality}
Let $(\X,\sfd,\mm)$ be an $\RCD^*(K,N)$ space with $K \in \R$ and $N < \infty$. Then, given any supersolution $(\phi_t)$ of the backward Hamilton-Jacobi-Bellman equation in the sense of Definition \ref{def:hjb} and any solution $(\nu_t,v_t)$ of the forward Fokker-Planck equation in the sense of Definition \ref{def:fpe} with the same parameter $c$, it holds
\begin{equation}\label{eq:duality1}
\int\phi_1\,\d\nu_1 - \int\phi_0\,\d\nu_0 \leq \frac{1}{2}\int_0^1\int |v_t|^2\,\d\nu_t\dt.
\end{equation}
Analogously, for any supersolution $(\tilde{\phi}_t)$ of the forward Hamilton-Jacobi-Bellman equation in the sense of Definition \ref{def:hjb} and any solution $(\tilde{\nu}_t,\tilde{v}_t)$ of the backward Fokker-Planck equation in the sense of Definition \ref{def:fpe} with the same parameter $c$, we have
\begin{equation}\label{eq:duality2}
\int\tilde{\phi}_0\,\d\tilde{\nu}_0 - \int\tilde{\phi}_1\,\d\tilde{\nu}_1 \leq \frac{1}{2}\int_0^1\int |\tilde{v}_t|^2\,\d\tilde{\nu}_t\dt.
\end{equation}
In particular, with the same assumptions and notations as in Section \ref{sec:2.2}, for any $\eps>0$ the following duality formula holds:
\begin{subequations}
\begin{align}
\label{eq:hjb-backward}
\eps\mathscr{I}_\eps(\mu_0,\mu_1) & = \eps H(\mu_0\,|\,\mm) + \sup\bigg\{\int \phi_1\,\d\mu_1 - \int \phi_0\,\d\mu_0 \bigg\} \\
\label{eq:hjb-forward}
& = \eps H(\mu_1\,|\,\mm) + \sup\bigg\{\int \tilde{\phi}_0\,\d\mu_0 - \int \tilde{\phi}_1\,\d\mu_1 \bigg\}
\end{align}
\end{subequations}
where the supremum is taken among all supersolution of the backward (resp.\ forward) Hamilton-Jacobi-Bellman equation in the sense of Definition \ref{def:hjb} with $c = \eps/2$ in \eqref{eq:hjb-backward} (resp.\ \eqref{eq:hjb-forward}).
\end{Theorem}

\begin{proof}
In order to prove \eqref{eq:duality1} let $(\phi_t)$ and $(\nu_t,v_t)$ be as in the statement, fix $x \in \X$, $R>0$ and let $\nchi_R$ be a cut-off function as in Lemma \ref{lem:cutoff}: by Definition \ref{def:hjb} it follows that $(\nchi_R\phi_t) \in AC([0,1],L^2(\X)) \cap L^\infty([0,1],W^{1,2}(\X))$ with $(\Delta(\nchi_R\phi_t)) \in L^\infty([0,1],L^2(\X))$. Thus Lemma \ref{lem:dermiste} applies to $(\nu_t)$ and $t \mapsto \nchi_R\phi_t$ on $[0,1]$, whence
\[
\frac{\d}{\d s}\Big(\int\nchi_R \phi_s\,\d\nu_s\Big)\restr{s=t} = \int \nchi_R\big(\frac{\d}{\d s}\phi_s\restr{s=t}\big)\d\nu_t + \frac{\d}{\d s}\Big(\int\nchi_R \phi_t \d\nu_s \Big)\restr{s=t}
\]
for a.e.\ $t \in [0,1]$. For the first term on the right-hand side, the fact that $\phi_t$ is a supersolution of the backward Hamilton-Jacobi-Bellman equation yields
\[
\int \nchi_R\big(\frac{\d}{\d s}\phi_s\restr{s=t}\big)\d\nu_t \leq -\int\nchi_R\Big(\frac{|\nabla \phi_t|^2}{2} + c\Delta \phi_t\Big)\d\nu_t.
\]
On the other hand, the fact that $(\nu_t,v_t)$ is a solution of the forward Fokker-Planck equation and $\nchi_R \phi_t \in D(\Delta)$ imply that
\[
\begin{split}
\frac{\d}{\d s}\Big(\int\nchi_R \phi_t\,\d\nu_s \Big)\restr{s=t} & = \int\Big(\nchi_R\langle\nabla \phi_t, v_t\rangle + \phi_t \langle\nabla\nchi_R,v_t\rangle\Big)\d\nu_t \\ & \, \quad + c\int\Big(\nchi_R\Delta \phi_t + 2\langle\nabla\nchi_R,\nabla\phi_t\rangle + \phi_t \Delta\nchi_R\Big)\d\nu_t
\end{split}
\]
and by Young's inequality
\[
\langle\nabla\phi_t,v_t\rangle \leq \frac{1}{2}|\nabla\phi_t|^2 + \frac{1}{2}|v_t|^2.
\]
From these observations and integrating over $[0,1]$ we obtain
\[
\begin{split}
\int\nchi_R \phi_1\,\d\nu_1 - \int\nchi_R \phi_0\,\d\nu_0 & \leq \int_0^1\int\nchi_R \frac{|v_t|^2}{2}\d\nu_t\dt + \int_0^1\int \phi_t\langle\nabla\nchi_R,v_t\rangle\d\nu_t\dt \\ & \, \quad + c\int_0^1\int\Big(\phi_t\Delta\nchi_R + 2\langle\nabla\phi_t,\nabla\nchi_R\rangle \Big)\d\nu_t\dt
\end{split}
\]
and, as in the proof of Theorems \ref{thm:main1} and \ref{thm:main1bis}, we can pass the limit $R \to \infty$ under the integral sign by dominated convergence: indeed, all the integrability properties of $\varphi_t^{\eps,\delta},\psi_t^{\eps,\delta},\vartheta_t^{\eps,\delta}$ that we used are still true for $\phi_t$ by Definition \ref{def:hjb}. Keeping in mind that $\nchi_R \to 1$, $|\nabla\nchi_R|,\Delta\nchi_R \to 0$ $\mm$-a.e.\ as $R \to \infty$ and $\nchi_R,|\nchi_R|,\Delta\nchi_R$ are uniformly bounded in $L^\infty(\X)$ w.r.t.\ $R$, \eqref{eq:duality1} follows.

As concerns \eqref{eq:hjb-backward} the `$\geq$' inequality is a direct consequence of \eqref{eq:duality1} and \eqref{eq:bbs2}. For the opposite inequality, notice that $(\psi_t^{\eps,\delta})$ as defined in the proof of Theorem \ref{thm:main1} is a solution to the backward Hamilton-Jacobi-Bellman equation in the sense of Definition \ref{def:hjb} only on the compact subsets of $[0,1)$. Thus let $\delta,s > 0$ and put
\[
\phi_t^{\delta,s} := \eps\log(\h_s g_t^\eps + \delta) = \eps\log(\h_{\eps(1-t)/2 + s} g^\eps + \delta).
\]
By Lemma \ref{lem:hjb} $(\phi_t^{\delta,s})$ is now a solution to the backward Hamilton-Jacobi-Bellman equation in the sense of Definition \ref{def:hjb} on the whole $[0,1]$, whence
\[
\int\phi_1^{\delta,s}\,\d\mu_1 - \int\phi_0^{\delta,s}\,\d\mu_0 \leq \sup\bigg\{\int \phi_1\,\d\mu_1 - \int \phi_0\,\d\mu_0 \bigg\}, \qquad \forall\delta,s > 0,
\]
the supremum being considered as in \eqref{eq:hjb-backward}. The continuity of $s \mapsto \phi_t^{\delta,s}$ in $L^2(\X,e^{-V}\mm)$ as $s \downarrow 0$ for all $t \in [0,1]$ (with $V = M\sfd^2(\cdot,x)$ for some $x \in \X$ and $M>0$) together with the fact that $\mu_0,\mu_1$ have compact support entails
\[
\lim_{s \downarrow 0}\int\phi_i^{\delta,s}\,\d\mu_1 = \int\psi_i^{\eps,\delta}\,\d\mu_i, \qquad i=0,1
\]
and arguing as in Theorem \ref{thm:main1} we can pass to the limit as $\delta \downarrow 0$, thus getting
\[
\int\psi_1^\eps\,\d\mu_1 - \int\psi_0^\eps\,\d\mu_0 \leq \sup\bigg\{\int \phi_1\,\d\mu_1 - \int \phi_0\,\d\mu_0 \bigg\}.
\]
Now it is sufficient to use the identity $\psi_0^\eps = -\varphi_0^\eps + \eps\log\rho_0$ in $\supp(\mu_0)$ together with \eqref{eq:entcost} to conclude.

By reversing time and following the same strategy, \eqref{eq:duality2} and \eqref{eq:hjb-forward} also follow.
\end{proof}

It is not difficult to deduce from the previous result that the entropic cost admits a Kantorovich-like dual representation, the Hopf-Lax semigroup being replaced by \eqref{eq:semigroup} suitably rescaled.

\begin{Theorem}[Kantorovich duality for the entropic cost]\label{thm:main2}
With the same assumptions and notations as in Section \ref{sec:2.2}, for any $\eps>0$ the following duality formula holds:
\[
\begin{split}
\eps\mathscr{I}_\eps(\mu_0,\mu_1) & = \eps H(\mu_0\,|\,\mm) + \sup_{u \in \mathbb{V}}\bigg\{\int u\,\d\mu_1 - \int Q_1^\eps u\,\d\mu_0 \bigg\} \\
& = \eps H(\mu_1\,|\,\mm) + \sup_{u \in \mathbb{V}}\bigg\{\int u\,\d\mu_0 - \int Q_1^\eps u\,\d\mu_1 \bigg\}
\end{split}
\]
where $\mathbb{V} := \{u : \X \to \overline{\R} \,:\, e^{u/\eps} \in L^2 \cap L^\infty(\X)\}$, with the convention $e^{-\infty} = 0$, and
\[
Q_1^\eps u := \eps\log\big(\h_{\eps/2} e^{u/\eps}\big).
\]
\end{Theorem}

\begin{proof}
Let us prove the second duality formula, as for the first one the argument is analogous. The `$\leq$' inequality is a trivial consequence of \eqref{eq:entcost}, the identity $\varphi_1^\eps + \psi_1^\eps = \eps\log\rho_1$ in $\supp(\mu_1)$ and the facts that $\varphi_0^\eps \in \mathbb{V}$, $\varphi_1^\eps = Q_1^\eps\varphi_0^\eps$. For the converse inequality, let $\delta,s > 0$ and define for all $t \in [0,1]$ and $u \in \mathbb{V}$
\[
Q_t^{\eps,\delta,s}u := \eps\log\big(\h_{\eps t/2 + s}e^{u/\eps} + \delta\big).
\]
By Lemma \ref{lem:hjb} $(Q_t^{\eps,\delta,s}u)$ is a solution to the forward Hamilton-Jacobi-Bellman equation in the sense of Definition \ref{def:hjb} on $[0,1]$, so that by \eqref{eq:hjb-forward}
\[
\eps\mathscr{I}_\eps(\mu_0,\mu_1) \geq \eps H(\mu_1\,|\,\mm) + \int Q_0^{\eps,\delta,s}u\,\d\mu_0 - \int Q_1^{\eps,\delta,s}u\,\d\mu_1.
\]
Then let us pass to the limit as $s \downarrow 0$ as just done in the proof of Theorem \ref{thm:duality} and, by monotonicity and the very definition of $\mathbb{V}$, as $\delta \downarrow 0$ too, thus getting
\[
\eps\mathscr{I}_\eps(\mu_0,\mu_1) \geq \eps H(\mu_1\,|\,\mm) + \int u\,\d\mu_0 - \int Q_1^\eps u\,\d\mu_1.
\]
Since this is true for all $u \in \mathbb{V}$, we conclude.
\end{proof}

\bibliographystyle{siam}
\bibliography{biblio}

\end{document}